\documentclass[12pt,draftcls,onecolumn]{IEEEtran}
\usepackage{setspace}
\usepackage{fancyhdr}
\usepackage{graphicx,graphics}
\usepackage{amsmath, amsthm, amssymb}
\usepackage{anysize}
\usepackage{subfig}
\usepackage{multirow}
\usepackage{array}
\usepackage{xcolor}
\usepackage{url}
\usepackage{epstopdf}
\marginsize{2cm}{2cm}{2cm}{2cm} 

\newtheorem{thm}{Theorem}[section]
\newtheorem{prop}[thm]{Proposition}

\theoremstyle{definition}

\numberwithin{equation}{section}

\newcommand{\R}{\mathbb{R}}

\usepackage{caption}

\newcommand\Tstrut{\rule{0pt}{2.6ex}}         
\newcommand\Bstrut{\rule[-0.9ex]{0pt}{0pt}}   

\begin{document}

\title{Adaptive Superresolution in Deconvolution of Sparse Peaks}

\author{Alexandra~Koulouri,  Pia~Heins
        and~Martin~Burger
\thanks{A. Koulouri  was with the Faculty of Information Technology and Communication Sciences,
Tampere University, P.O. Box 692, 33101 Tampere, FI,  e-mail: alexandra.koulouri@tuni.fi.}
\thanks{Pia Heins was with Faculty of Electrical Engineering and Information Technology, University of Applied Sciences and Arts Hannover, Ricklinger Stadtweg 120, 30459 Hannover, DE e-mail: pia.heins@hs-hannover.de
}
\thanks{Martin Burger was with the Department of Mathematics,
Universit\"at Erlangen-N\"urnberg, Cauerstrasse 11, 91058 Erlangen,
DE, e-mail: martin.burger@fau.de}
}

%

\maketitle

\begin{abstract}
The aim of this paper is to investigate superresolution in
deconvolution driven by sparsity priors. The observed signal is a
convolution of an original signal with a continuous kernel. With the
prior knowledge that the original signal can be considered as a
sparse combination of Dirac delta peaks, we seek to estimate the
positions and amplitudes of these peaks by solving a finite
dimensional convex problem on a computational grid. Because, the
support of the original signal may or may not be on this grid, by
studying the discrete de-convolution of sparse peaks using
$\ell_1$-norm sparsity prior,
we confirm recent observations that canonically the discrete
reconstructions will result in multiple peaks at grid points
adjacent to the location of the true peak. Owning to the complexity
of this problem, we analyse carefully the de-convolution of single
peaks on a grid and gain a strong insight about the dependence of
the reconstructed magnitudes on the exact peak location. This in
turn allows us to infer further information on recovering the
location of the exact peaks i.e. to perform super-resolution. We
analyze in detail the possible cases
that can appear and 
based on our theoretical findings, we propose an self-driven
adaptive grid approach that allows to perform superresolution in
one-dimensional and multi-dimensional spaces. 
With the view that the current study can provide a further step in
the development of more robust algorithms for the detection of
single molecules in fluorescence microscopy or identification of
characteristic frequencies in spectral analysis, we demonstrate how
the proposed approach can recover sparse signals using simulated
clusters of point sources (peaks)
of low-resolution in one and two dimensional spaces. 

\end{abstract}



\section{Introduction}

\subsection{Deconvolution of sparse peaks on discrete  grids}
In a wide range of imaging applications, the signal of interest
comprises a sequence of sparse peaks (or point sources) for instance
in fluoresce microscopy \cite{Jones2011,Falcon,Small2014}, astronomy
\cite{Li2018}, ultrasound or Doppler technology
\cite{Tur2011,Wagner2012,Bar-Ilan2014,Bendory2016a}, medical imaging
\cite{Tournier2004,PIA2015} and  computational neuroscience
\cite{Ekanadham2014}. In these applications, one has frequently to
solve the problem that the signal of interest cannot be observed
directly, but has to be inferred from other quantities, often low
spatial resolution observations which mathematically can be
described as the convolution of the
original signal with a smooth kernel.

In this article, we study the superresolution problem, known as
sparse peak deconvolution \cite{Duval2015b}, where one seeks to
estimate the positions and amplitudes of the underlying sparse peaks
from a set of blurred observations. As idealized data we consider
the convolution of a measure $\mu$ on $\Omega\subseteq \mathbb{R}^d$
with a known symmetric and smooth (with infinite support) kernel $G$
which attains its maximum at $0$, i.e.
\begin{equation}\label{eq:convolution}
    f(x) = (G*\mu)(x):=\int_{\Omega} G(x-y)~d\mu(y)\;.
\end{equation}
Here we consider the convolution operator from ${\cal M}(\Omega)$ to
$L^2(\Omega)$, which is well-defined by the Fourier convolution
theorem  (cf. \cite{champeney1989handbook}) and 
we are interested in the reconstruction of sparse peaks when the
corresponding original signal is of the form
\begin{equation}\label{eq:originalSignal}
    \mu(x)= \sum_{l=1}^L \gamma_l \delta_{\xi_l}\;,
\end{equation}
where $L$ is the total number of peaks and
$\delta_{\xi_l}=\delta(x-\xi_l)$ denotes a concentrated measure
(expressed through the Dirac-delta function $\delta$) at location
$\xi_l:\Omega\rightarrow \R^d$ with amplitude $\gamma_l$.

In order to obtain a sparse reconstruction it is nowadays standard
to employ the well-established $\ell_1$-norm minimization approaches
(also known as Basic Pursuit or LASSO)
\cite{Donoho1992,Tibshirani1996,Chen1998}
which, in addition to sparse promoting solutions, allow the
linearization of the original problem, the direct application of
fast convex optimization solvers (e.g. \cite{Grant2014}) and do not
require application of Fourier transform\cite{Candes2014}. So,
instead of solving a continuous deconvolution problem
\cite{Decastro2012,Candes2013,Bredies2013,Candes2014}, the aim is to
reconstruct $\mu$
via a discrete set of concentrated measures, i.e 
to look for a discrete solution of the form
\begin{equation}\label{eq:numericalSolution}
    \mu^N(x) = \sum_{k=1}^N c_k \delta_{x_k}\;,
\end{equation}
where $ c=\{c_k\}_{k=1,\ldots,N}\in\mathbb{R}^N$ is a vector that
contains the numerically estimated amplitudes (weights) at a set of
grid points $\{x_k\}_{k=1,\ldots,N}$.
With the discretization of the computational domain, the convolution
can be written as an operator acting on the coefficients $
c=\{c_k\}_{k=1,\ldots,N}$, i.e.
\begin{equation}
    G*\mu^N  = Ac = \sum_{k=1}^N c_k G(x-x_k)\;,
\end{equation}
where $A:\R^N \rightarrow L^2(\Omega)$. The $\ell_1$-norm minimization problem is
\begin{equation}\label{eq:deconvolutionProblem}
    \min_{c \in \R^N} \ J(c):=\frac{1}{2}\|Ac-f\|^2+\lambda \|c\|_1\;.
\end{equation}
Since the support of signal $\mu$ \eqref{eq:originalSignal} may
or may not be on the computational grid  $\{x_k\}_{k=1,\ldots,N}$,
three chief questions for the $\ell_1$-norm estimates arise:
\begin{itemize}
  \item How the error between the original signal $\mu$ and discrete signal $\mu^N$
is quantified based on the discretization?
 \item  What are the expected patterns of the discrete estimation $\mu^N$ on
an arbitrary grid?
\item Can  the locations and amplitudes of the
  original signal $\mu$ be approximated with the help of the
solution
$\mu^N$?
\end{itemize}
In this article, we investigate and answer these questions with the
help of convex optimization theory and standard numerical analysis.
We anticipate that the understanding of the effects of sparsity
promoting solvers when the computational grid and the support of the
original sparse signal do not coincide will allow the development of
more robust algorithms required in applications such as fluorescence microscopy 
\cite{Elson2002,Sage2015,Rust2006,Hess2006,Gould2009,Small2014,Holden2011,ThunderSTORM2014}.

\subsection{Related works}
Sparsity prior driven deconvolution approaches in  continuous
domains have been  studied in several works including
\cite{Decastro2012,Candes2013,Bredies2013,Duval2015,Fernandez-Granda2016}.
The signal (sum of Dirac functions) to be recovered is not a
finite-dimensional vector (as in \eqref{eq:numericalSolution}) but
 a Radon measure and 
the minimization problem is formulated with the help of the total
variation (TV) term and this problem is referred to as
Beurling-Lasso (BLASSO) \cite{Decastro2012,Duval2015,Poon2019}. An
extensive theoretical analysis of BLASSO in the case of one
dimensional Fourier measurements was provided in \cite{Candes2014}.
Particularly, it was shown that if the spikes are separated enough,
then the exact recovery is possible (when the fraction of the
measurement noise and regularization parameter tends to zero).
Robustness to noise under this separation condition was studied by
\cite{Azais2015,Duval2015}, while  the effect of the positivity
constraint was analyzed in
\cite{Bendory2016,Morgenshtern2016,Denoyelle2016,Eftekhari2018a,Eftekhari2018}.
The observation sampling and the exact support recovery was
theoretically studied in \cite{Duval2017}. In \cite{Poon2019}, the
BLASSO problem was analyzed for measures in higher dimensional
spaces revealing that the kernel and arrangement of the original
peaks affect the stability in the estimates.

BLASSO is a convex but infinite dimensional optimization problem. As
shown in
\cite{Candes2014,Bhaskar2013,Tang2013,Bendory2015a,Fernandez-Granda2016},
solvers exist for ideal lowpass filters (i.e. Dirichlet type of
kernels) when the observations are transformed into the Fourier
domain and consider a finite number of frequencies in one
dimensional signal spaces. Particularly, in theses cases, the primal
(BLASSO) problem is expressed via its (Fenchel-Rockafellar)
associated finite dimensional dual problem, which for the numerical
computations is encoded as a finite semi-definite program (SDP)
\cite{Candes2014,Castro2017}. The core of these approaches rely on
the duality between peak locations and the existence of an
interpolating trigonometric polynomial (often referred to as dual
certificate) in the measurement (dual) space (which is bounded by 1
in magnitude at locations indicating the underlying peaks
\cite{Candes2006}). However,
apart from  the one dimensional spaces (line and torus), there is
not a canonical extension or exact SPD formulation in higher
dimensional
spaces.
We refer for instance to \cite{Castro2017} (and the references
therein) for  relaxed SDP versions in higher dimensions.

In arbitrary spaces and for general kernels, one has to approximate
the BLASSO problem by first introducing discrete grids and then
solve a finite
dimensional minimization problem (e.g. LASSO or basic pursuit). 
Several authors have proposed approximation or non-convex
optimization steps to be included in the standard LASSO to recover
the exact locations and amplitudes. In particular, in
\cite{Ekanadham2011,Duval2015b}, the continuous basis-pursuit which
involves a first order Taylor approximation of the kernel in the
fidelity term of the minimization problem accompanied by the
$\ell_1$-norm regularization term have been utilized to improve the
accuracy in the peak localization. However, multiple peaks around
the original peak is a common result (as also theoretically
justified in \cite{Duval2015b}). More robust iterative approaches
 using the Frank-Wolfe algorithm (also known as conditional gradient method) has
been proposed in \cite{Bredies2013,Boyd2017}. These include an
alternation between two steps. In the first step, the computational
support is renewed by generating a new peak location using the
conditional gradient method  and then a non-convex step follows
where only the locations and amplitudes are computed while the
number of peak locations stays fixed.

Even though the current paper is focusing on superresolution using
convex optimization methods and especially the $\ell_1$-norm
regularization, we would like to mention that  there is also a vast
literature on spectral superresolution algorithms that rely on Prony's
concept (for a general review see \cite{Krim1996,Stoica2005}), for
example MUSIC \cite{Schmidt1986}, ESPRIT \cite{Roy1989}) or pencil
method \cite{Hua1990}.
These methods perform well in noiseless setting and do not require a
minimum separation condition to fully recover positive and negative
peaks; however they rely strongly on the signal, noise and
measurement modelling and their extension to higher dimension is not
trivial  see e.g.
\cite{Liao2015,Peter2015,Kunis2016,Andersson2018,Li2019,Diederichs2018}.
%
\subsection{Contributions}
In this work, the aim is to find new connections between the
super-resolution algorithms which impose sparsity assumptions on the
signal to be recovered \cite{Small2014} and theoretical studies
(e.g. \cite{Morgenshtern2016,Bendory2017,Duval2015b}) which have
been developed rather separately so far.
To that end, we first explain  how convex optimization techniques
\cite{Boyd2004,Hindi2006,Boyd2011} and, more precisely, $\ell_1$-
norm sparsity constraints affect the solution of such inverse
problems as the deconvolution of sparse peaks (or point sources) on
discrete grids (or meshes) when the convolution kernel is smooth
(admissible) \cite{Bendory2016} and then we propose an adaptive
super-resolution scheme. In particular, our contributions are two-fold and are summarized as follows:
\subsubsection{Theoretical}
\begin{itemize}
 \item With the help of the first order optimality condition of the
$\ell_1$-norm minimization problem, we show that the numerical
solution consists of one or multiple peaks at grid points (or nodes)
adjacent to the location of the actual peak. Our conclusions are
inline with recent results presented in \cite{Duval2015b} for one
dimensional spaces; but, the methodology employed, as well as the
form in which the problems in question are expressed, are different.
Previous approaches study  the properties of the $\ell_1$-norm
numerical solution by
 introducing the extended computational support notion \cite{Duval2015b} or by
 deriving
dual certificates that fulfils particular properties \cite{Poon2018}
which both were used nicely to obtain asymptotic properties of the
signal support. Here, we take a step forward to characterize the
values of the numerical solution on its support also in dependence
of the exact peak locations.
We use the optimality
condition of the finite dimensional $\ell_1$-norm minimization problem to
investigate the conditions under which a single or multiple peaks
are recovered in one dimensional spaces. Then, we define an
optimality curve, directly related to the optimality condition of
the problem (respectively the dual certificate), whose shape allows
us not only to justify the patterns of the expected numerical
solutions on fixed computational grids both in one and higher
dimensional spaces, but also infer further information on the
location of the exact peaks.
\item We show that we can explicitly approximate
 the locations and amplitudes of the exact peaks based on a set of linear
equations
 derived from the associated normal equations of the $\ell_1$-norm
 problem.
  \item We derive an a-posterior error between
the original signal \eqref{eq:originalSignal} and its discrete
version \eqref{eq:numericalSolution} by employing the Bregman
distance \cite{Burger2016}. We show that the recovery error depends on the relative distance
between the computational
grid points and the locations of the original peaks.
\end{itemize}
\subsubsection{Practical} The a-posterior error outcome and the numerical
reconstructions of multiple peaks at grid points in the vicinity of
the original peaks give us the intuition to introduce the adaptive
grid concept for the recovery of the original peaks. Hence, we
propose an adaptive super-resolution scheme
consisting of two main stages. 
 First we determine the intervals which include the
support of the original peaks and we separate multiple original
peaks which are close to each other. This is achieved by adjusting
the grid as the computations proceed in a manner dependent upon the
previous sparse solution.
 Then, the coordinates of the locations and the amplitudes of the peaks are
approximated based on the numerical solution obtained from the first
stage and the set of equations following from the optimality
condition of the formed $\ell_1$-norm minimization problem.

The proposed adaptive algorithm shares some similarities with other
superresolution algorithms e.g. \cite{Falcon,Zhu2012}; however,
our approach embeds an automatic adaptation scheme \cite{Trottenberg2000} since 
it  restricts and refines the grid in an unsupervised manner only in
the neighborhoods where there is indication that a peak exists.
This additionally allows to solve  iteratively a small to medium size
linear problem using convex optimization techniques
\cite{Boyd2004}, instead of a big size problem as in \cite{Falcon}.

Overall, our analysis provides  theoretical insights on the effects
of gridding (a.k.a discetization of the parameter space), and it can help the algorithmic development  in the
direction of avoiding heuristic post-processing steps  by using
information about the convolution kernel properties and the
formulation of the minimization problem rather than  resorting in
unreliable approximations as often happens in application papers (e.g. in \cite{Tang2013a}).

\section{deconvolution of sparse peaks by convex
optimization}\label{sec:ContinuousVsDiscrete}

In the following we discuss the theoretical basis of sparse peak
deconvolution using convex optimization approaches and put it in
perspective with classical discretization issues in numerical
analysis.

\subsection{Sparsity over the continuum and its discretization}
Let us start by formulating the problem  over the continuum, following
\cite{Bredies2013,Duval2015}, which is the underlying ideal sparse
peak deconvolution to which we expect minimizers of
\eqref{eq:deconvolutionProblem} to converge to. For a Radon measure
$\mu$ on $\Omega$ we denote its total variation by
\begin{equation}
    \Vert \mu \Vert_{TV} = \sup_{\varphi \in C_0(\Omega)} \int_\Omega \varphi(x) ~d\mu(x).
\end{equation}
The convex variational problem solved for sparse peak deconvolution
in a continuum setting is then given by
\begin{equation} \label{eq:deconvolutionProblemContinuum}
    J_\infty(\mu) = \frac{1}2 \Vert G*\mu - f \Vert^2 + \lambda \Vert \mu \Vert_{TV}.
\end{equation}
Now \eqref{eq:deconvolutionProblem} can be interpreted as a
discretization on a given grid, it can indeed be rephrased as
\begin{equation}
    J_N(\mu) = \left\{ \begin{array}{ll}\frac{1}2 \Vert G*\mu - f \Vert^2 + \lambda \Vert \mu \Vert_{TV} & \text{if } \mu \in \text{ span}(\{\delta_{x_k}\}_{k=1,\ldots,N}) \\ +\infty & \text{else.}\end{array} \right.
\end{equation}
It is straight-forward to show that the functionals $J_N$
$\Gamma$-converge to $J_\infty$, but one can also ask for more
quantitative error estimates, which we shall discuss below.

By standard arguments we can verify the following result for the
discretized problem (cf. \cite{PIA2015,Bredies2013,Duval2015} for
analogous results on the continuum problem
\eqref{eq:deconvolutionProblemContinuum}):
\begin{prop}\label{prop:existence}
For $\lambda \geq 0$ there exists a solution of
\eqref{eq:deconvolutionProblem}. If $\lambda > \|A^*f\|_\infty$ then
the unique minimizer is given by $c = 0$. If $\lambda <
\|A^*f\|_\infty$, then each solution is different from zero.
\end{prop}
\begin{proof}
Convexity, coercivity, and nonnegativity immediately imply the
existence of a minimizer. Now let $\lambda > \|A^*f\|_\infty$, then
\begin{eqnarray*} J(c)&=&\frac{1}2 \| A c - f\|^2 + \lambda  \|c\|_1  \\ &=&  \frac{1}2 \|Ac\|^2 + \frac{1}2\|f\|^2 + \lambda  \|c\|_1  -
\langle A^* f, c\rangle \\ &\geq& \frac{1}2 \|Ac\|^2 +
\frac{1}2\|f\|^2  + (\lambda-\|A^*f\|_\infty)  \|c\|_1  \geq
\frac{1}2\|f\|^2 \\&=&J(0),
\end{eqnarray*}
with inequality only for $c=0$. Hence, $c = 0$ is the unique
minimizer. In the case $\lambda < \|A^*f\|_\infty$ we choose $c=
\epsilon A^*f $ with $\epsilon > 0$ sufficiently small to verify
that there exists a $c$ yielding a functional value lower than
$\frac{1}2\|f\|^2$.
\end{proof}

\subsection{Optimality conditions}
As a next step we state the optimality conditions for
\eqref{eq:deconvolutionProblemContinuum} and the discrete version
\eqref{eq:deconvolutionProblem}. Those are important for error
estimates and further analysis in this paper.

Let us start with the sub-differential of the total variation norm,
which is given by (cf. \cite{Bredies2013})
\begin{equation}
    \partial \Vert \mu \Vert_{TV} = \{ q \in L^\infty(\Omega)~|~ \Vert q \Vert_\infty \leq 1, q(x) \equiv \pm 1 \text{ on supp}(\mu_\pm) \}.
\end{equation}
Here $\mu = \mu_+ - \mu_-$ denotes the standard Jordan decomposition
of the signed measure $\mu$. Since the quadratic part of the
functional $J_\infty$ is differentiable and $G$ is continuous, i.e.
the convolution maps into the pre-dual of the space of Radon measures, we obtain the optimality condition
\begin{align}\label{eq:opt_TV}
&\Vert G\ast f-H\ast \mu \Vert_\infty \leq \lambda  \\
& G\ast f  - H\ast \mu  = \pm \lambda \quad \text{in supp}(\mu_\pm).
\end{align}
where $H=G\ast G$.
 On the other hand, the optimality condition of the discrete
problem \eqref{eq:deconvolutionProblem} is
\begin{equation}\label{eq:opt_l1_orig}
    \lambda p_j = [A^*(f-A c)]_j\quad\quad\mbox{for}~j=1,\ldots,N,
\end{equation}
\noindent where $p\in\mathbb{R}^N$ is contained in the
sub-differential of $\|c\|_1$. The right hand side of the previous
equation\footnote{ The convolution of signal $\mu^N$ with a kernel,
e.g. Gaussian $G$, is $G*\mu^N=\sum_{k=1}^Nc_kG(x-x_k)$ and in
matrix form this can be expressed as
$[Ac]_j=\sum_{k=1}^Nc_kG(x_j-x_k)$. Moreover, $G*G = H$ and since
function ${G}$ is symmetric and the convolution is associative
$G\ast(G\ast \mu^N) = (G*G)* \mu^N=H\ast \mu^N
=\sum_{k=1}^Nc_kH(x-x_k) .$} is
$[A^\mathrm{T}f]_j=\sum_{l=1}^L\gamma_l H(x_j-\xi_l)$ and
$[A^\mathrm{T}Ac]_j=\sum_{k=1}^Nc_iH(x_j-x_k)$. The optimality
condition can be written as
\begin{equation}\label{eq:opt_l1}
    \lambda p_j =\sum_{l=1}^L\gamma_iH(x_j-\xi_l)-\sum_{k=1}^N c_k H(x_j-x_k),
\end{equation}
\noindent for $k=j$,  $p_j\in sign(c_j)$ when $c_j\neq 0$ and $\vert
p_j \vert<1$ when $c_j=0$.

In order to highlight the connection with the continuum formulation,
we rewrite the optimality solely for the measure $\mu^N$ and deduce
that
\begin{align}\label{eq:opt_l1_generic}
&G\ast f-\vert H\ast \mu^N\vert  \leq \lambda \quad \text{in } \{x_k\}_{k=1,\ldots,N} \\
&  G\ast f - H\ast \mu^N  = \pm \lambda \quad \text{in
supp}(\mu^N_\pm).
\end{align}
We see that the main difference to the optimality condition in the
continuum is that the first equality only holds on the grid points
$x_k$ and not in the whole domain $\Omega$. Note that due to the
continuity of $G$ and $H$ one will expect (at least for sufficiently
small grid size) that if  $\frac{\vert H\ast \mu^N - G\ast
f\vert}{\lambda}-1$ is strictly less than zero in a set of
neighbouring grid points, then it remain less than zero in the area
bounded by these points (further details are given in
section~\ref{subsection:offgridPeaks}). Hence, the main violation of
the continuum optimality condition considered for $\mu^N$ will
appear close to grid points where it equals zero, usually
corresponding to non-zero coefficients $c_k$. This yields a first
idea for using an adaptive computational grid. As we shall see below
this can be further improved and backed up by a-posteriori error
estimation.

\subsection{A-Posteriori error estimate}\label{section error estimation}
In order to derive suitable error estimates for non-smooth convex
variational problems such as problem
\eqref{eq:deconvolutionProblem}, it is now a standard approach to
use the Bregman distance as proposed in \cite{Burger2004} (we refer
to \cite{Burger2016} for an overview). The Bregman distance for the
total variation distance is given by
\begin{equation}
    D^q_{TV}(\tilde \mu,\mu) = \Vert \tilde \mu \Vert_{TV} - \Vert \mu \Vert_{TV} - \langle q, \tilde \mu - \mu \rangle
\end{equation}
for a subgradient $q \in \partial \Vert \mu \Vert_{TV}$. Given a
subgradient $\tilde q \in \partial \Vert \tilde \mu \Vert_{TV}$, we
will denote by
\begin{equation}
    D^{\tilde q,q}_{TV} = D^{\tilde q}_{TV}(\mu,\tilde \mu) + D^q_{TV}(\tilde \mu,\mu) = \langle \tilde q - q, \tilde \mu - \mu \rangle
\end{equation}
the symmetric Bregman distance.

The key idea here is to use the difference in the optimality
conditions and take a duality product with the difference of the
measures. For this sake, we use the following notation
\begin{equation}\label{eq:subdiff_measures}
\begin{split}
    q^N(x) &:= \min\{ \max\{ \frac{G*f(x) - H*\mu^N(x)}\lambda,-1\},1\},\\ \qquad r^N(x) &:=     \frac{G*f(x) - H*\mu^N(x)}\lambda -q^N(x).
\end{split}
\end{equation}
It is straightforward to see that $q^N \in \partial \Vert \mu^N(x)
\Vert_{TV}$ and hence from \eqref{eq:opt_TV},
\eqref{eq:opt_l1_generic} and \eqref{eq:subdiff_measures} we
obtain
\begin{equation} H*(\mu - \mu^N) + \lambda (q - q^N) =
\lambda r^N. \end{equation}
Now the announced duality product with
$\mu-\mu^N$ implies an a-posterior error estimate of the form
\begin{equation}
    \Vert G*(\mu - \mu^N) \Vert^2 + \lambda D^{q,q^N}_{TV}(\mu,\mu^N) = \lambda \langle r^N, \mu - \mu^N \rangle.
\end{equation}
Thus, we observe that only regions with $r^N \neq 0$ contribute to
the error between $\mu$ and $\mu^N$.

Moreover, via the optimality condition \eqref{eq:opt_l1_generic} of problem
\eqref{eq:deconvolutionProblem},   we have that $r^N(x_k) = 0$ for any grid
point $x_k$ and thus  $\langle r^N, \mu^N \rangle = 0$. So, we can write
\begin{equation}
    \Vert G*(\mu - \mu^N) \Vert^2 + \lambda D^{q,q^N}_{TV}(\mu,\mu^N) \leq \lambda \Vert r^N \Vert_\infty \Vert \mu \Vert_{TV}.
\end{equation}
The previous expression shows that  by reducing the supremum norm of $r^N$ is crucial for reducing
the global error. This can be achieved by introducing finer
computational grids. In Figure~\ref{fig:PosteriorError} we can
observe that the value of $\|r^N\|_\infty$ decreases with respect to
the size of the computational grid.
\begin{figure}[!htb]
    \begin{center}
       \includegraphics[width=0.5\textwidth]{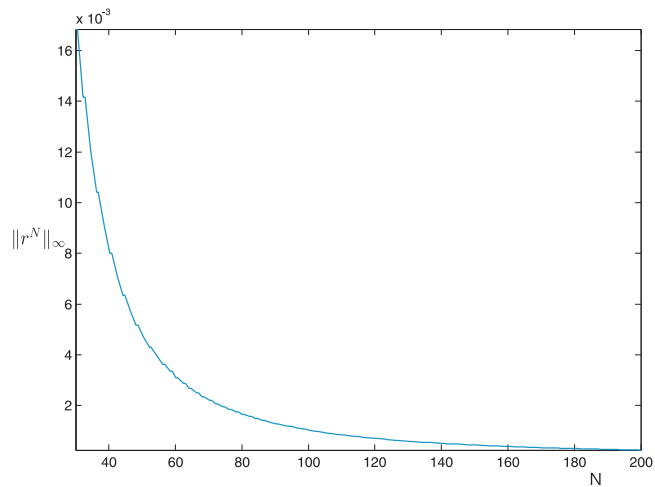}
       \caption{Supremum norm of $r^N$(residual) with respect to the size $N$ of the computational grid. Here, we considered $\mu = \delta_\xi$ and $\mu^N = \sum_{k=1}^N
c_k\delta_{x_k}$. Thus, $r^N(x) = \frac{H(x-\xi) - \sum_{k=1}^N c_k
H(x-x_k)}{\lambda}-q^N(x)$ based on \eqref{eq:subdiff_measures} and
$H(x)$ was a Gaussian kernel.}\label{fig:PosteriorError}
        \end{center}
    \end{figure}
However, we remark that by introducing fixed fine grids
(as proposed for instance in \cite{Zhu2012,Mukamel_2012,Falcon}) the
computational performance and stability can be affected.
To overcome computational limitations, we later propose to introduce a
progressive grid refinement.

\section{Numerical solutions in the cases of single
peaks}\label{sec:deconvolutionSinglePeak}
 Here, as started in \cite{heinsnovel} we
analyze in detail the solutions of the form
\eqref{eq:numericalSolution} in the case of noiseless data $f$
produced by a single positive peak, i.e.
\begin{equation}
    \mu = \gamma \delta_\xi, \qquad\mbox{and}\qquad f(x) = \gamma G(x-\xi)\;,
\end{equation}
where $\xi\in\mathbb{R}^d$ and $d\geq1$ \footnote{Our analysis is based on a single positive peak which is often the case in image processing applications. We note that  the theorems/conclusions presented in this section are valid also for a negative peak.}.
 We can easily interpret Proposition \eqref{prop:existence} in this
case as $A^*f = \gamma [H(x_j-\xi)]_{j=1,\ldots,N}$ and
$\|A^*f\|_\infty>\lambda$ in order to obtain nonzero solutions.
Hence, we need $\lambda < \gamma \,\underset{j}{\max}\{
H(x_j-\xi)\}$ and since $H$ attains its maximum at zero, a simple
sufficient condition is given by
\begin{equation} \label{eq:lambdacondition}
 \lambda < \gamma H(0),
\end{equation}
which is also necessary in the case of $\xi$ coinciding with a grid
point. We will thus assume condition \eqref{eq:lambdacondition}
throughout the whole section without further notice.

\subsection{Exact recovery}
The simplest case to start with, which can directly be treated in
arbitrary dimensions, is that $\xi$ coincides with one of the grid
points. In this case we obviously expect perfect reconstruction,
which is confirmed by the following result:
\begin{prop}
    Let be $\xi = x_K$ for some $K \in\{1,\ldots,N\}$.
    Then there exists a one-sparse solution $\mu^N$ of
    \eqref{eq:deconvolutionProblem}, which is nonzero at $x_K$, i.e.,
    $\mu^N = c_K \delta_{x_K}$ with $c_K =  \frac{\gamma H(0)-\lambda}{H(0)} \in (0,\gamma)$.
    \label{thm:recon_spikes_on_gridpoints}
\end{prop}
\begin{proof}
Without restriction of generality assume that $\gamma > 0$. In order
to prove the assertion, we have to check whether the optimality
condition of \eqref{eq:deconvolutionProblem} holds under the
assumptions mentioned above. From our prior computations \eqref{eq:opt_l1}, the
optimality condition \eqref{eq:opt_l1} reduces to
\begin{equation}\label{eq:opt_simple_l1}
     \lambda p_j = \gamma  H(x_j-\xi) -  c_K H(x_j-x_K)  = (\gamma-c_K)H(x_j-x_K) .
\end{equation}
We have to differentiate between the cases where $j=K$ and
$j\neq K$.
\\
For $j=K$ the optimality condition \eqref{eq:opt_simple_l1} is
\begin{equation}\label{eq:ExactSupportSinglePeak}
      p_K = (\gamma-c_K)\frac{H(0)}{\lambda}  = 1\; .
\end{equation}
For $j\neq K$  we have
\begin{align*}
    p_j = \frac{\gamma-c_K}{\lambda}H(x_j-x_K) < (\gamma-c_K)\frac{H(0)}{\lambda} \; ,
\end{align*}
due to the fact that $H$ attains its maximum at zero. Hence in both
cases the optimality condition is fulfilled and we obtain the
assertion.
\end{proof}
Therefore, the reconstruction of the support of a delta peak is
exact if the position of the peak coincides with a grid point and
the regularization parameter is small enough.

\subsection{Recoveries for  off-the-grid
peaks}\label{sec:1Danalysis}
Let us consider the more frequent case where
$\mu=\gamma\delta_\xi$ is located among a set of grid points
$\mathcal{N}=\{x_k\}_{k=1:N}$.
Here, with the help of the optimality
conditions \eqref{eq:opt_l1_generic}, we define a so-called
optimality curve $p(x)$ given by
\begin{equation}p(x) =
\frac{G\ast f- H\ast\mu^N }{\lambda}-1,\end{equation}
where $f=G\ast\mu$, $H=G\ast G$ (smooth and symmetric) and
$\mu^N=\sum_{x_k\in\mathcal{N}}c_k\delta_{x_k}$
is the nontrivial numerical solution obtained from the minimization
problem~\eqref{eq:deconvolutionProblem}.
The optimality curve $p(x)$
can be regarded as analogous to the $TV$ dual certificate
\cite{Poon2019} for the $\ell_1$- norm minimization problem
that will allow us to understand the expected
patterns of the numerical solutions around $\xi$.

To ease our analysis, we rewrite $p(x)$ as
\begin{equation}\label{eq:p(x)}
p(x) =
\frac{1}{\lambda}H(x-\xi)-\frac{1}{\lambda}\sum_{x_k\in\mathcal{N}}c_kH(x-x_k)-1.
\end{equation}
We note that $p(x_j)<0$  when $c_j=0$ and $p(x_j) = 0$ when $c_j > 0$.
\subsubsection{Single spatial dimension} In one dimensional spaces, we consider that signal $\mu=\gamma\delta_\xi$ is located
between two grid points, i.e. $\xi\in(x_K,x_{K+1})$. For the
following consideration, the interval length $h$ will be defined as
$$ h:= |x_{K+1} - x_K| \; . $$

By employing function $p(x)$ \eqref{eq:p(x)} in the vicinity
of peak $\xi$, we can prove the following Theorem.
\begin{thm}\label{th:1D}
~\\
Let $H \in C^3(\mathbb{R})$ be nonnegative with a unique maximum at
zero and let $h$ be sufficiently small. \\ Assume $\mu=\gamma\delta_\xi$,
$\xi\in(x_K,x_K+\frac{h}{2})$ for $K\in\{1,\ldots,N-1\}$ and
    $\lambda<\gamma H(x_K-\xi)$ holds. \\
    When we have that $\xi\in(x_K,\,x_K+\frac{\lambda h}{2\gamma H(0)})$, there
    exists a solution of \eqref{eq:deconvolutionProblem}, which can be
    written as $\mu^N = a\delta_{x_K}$ with $a =
    \frac{\gamma H(x_K-\xi)-\lambda}{H(0)} \in (0,\gamma)$.\\
    Moreover, if we have $\xi\in(x_K+\frac{\lambda h}{2\gamma H(0)},\, x_K+\frac{h}{2})$,
    then $\mu^N = a\delta_{x_K}$ is \emph{not} a solution of
    \eqref{eq:deconvolutionProblem} for any $a\in\mathbb{R}^{+}$. Instead the solution is of the form
        $\mu^N = c_K \delta_{x_K} + c_{K+1} \delta_{x_{K+1}}$ with $c_K$ and $c_{K+1}$ being nonzero and of the same sign as $\gamma$.
    \label{thm:spikes_between_grid_points}
\end{thm}
The proof of Theorem~\ref{thm:spikes_between_grid_points} is given
in Appendix~\ref{pr:theorem1D}. Figure
\ref{fig:sparse_spikes_interval_tikz} illustrates the assertion of
Theorem \ref{thm:spikes_between_grid_points}. Note that due to the
symmetry of $H$, the analogous claim holds for $\xi$ in the other
half of the interval.
\begin{figure}[h!]
    \begin{center}
             \includegraphics[width=0.55\textwidth]{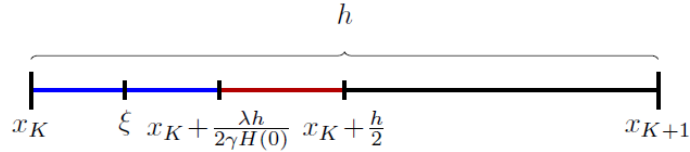}
        \caption[Reconstruction of sparse spikes]{If $\xi$ is in the blue interval,
        the reconstructed solution $\mu^N$ consists of only one peak. In the case
        that $\xi$ is located in the red interval, then one recovered peak is not sufficient.}
        \label{fig:sparse_spikes_interval_tikz}
    \end{center}
\end{figure}
Figure~\ref{fig:OptCurve}  depicts the optimality curve $p(x)$ for a positive peak when $H(x)$ is a Gaussian kernel.
The curve is downward concave in the area around $\xi$ which implies that there are at most two
points on
x-axis where $p(x)=0$. From these points, at least one is the grid point with the nonzero coefficient of
$\mu^N$.
 We can observe
that the number of the recovered peaks depends on the distance
between the location of the exact peak (denoted by red
$\mathrm{x}$) and the neighboring grid points (given fixed
$\lambda=0.1\lambda_\mathrm{max}$).
\begin{figure}[h!]
    \begin{center}
       \includegraphics[width=0.55\textwidth]{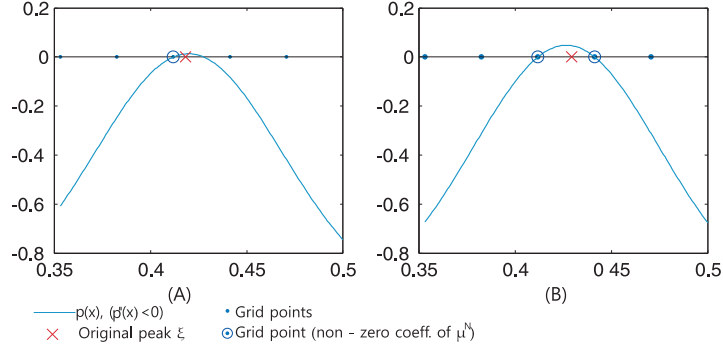}
       \caption{Function $p(x)$ around location $\xi$ for the cases where the reconstructed signal $\mu^N$ has (A) a single
nonzero coefficient and (B) two nonzero coefficients between the location of
the exact peak $\xi$.}
      \label{fig:OptCurve}
    \end{center}
\end{figure}
\subsubsection{Higher spatial dimension}\label{subsection:offgridPeaks}
In higher spatial dimensions, the topological structure is more complicated
which makes a rigorous proof by analogous arguments impossible. However,
we can at least make some formal arguments and computational experiments concerning
the optimality curve $p(x)$ \eqref{eq:p(x)}. First of all we expect that for $\gamma>0$,  $p(x)$ is concave around $\xi$ and  nonzero entries are only found in the convex hull of $\xi$ on the grid, i.e. the
largest convex hull $C_\xi$ that can be formed of a set of grid points surrounding
the peak $\xi$ such that no other grid point is contained in $C_\xi$.

The following observations can be made:
\begin{itemize}
  \item Given any location $x$ far from $\xi$, we have that $p(x)<0$, since
the positive term $H(x-\xi) $  is smaller than $H(x-x_k)$.
\item
Since the  kernel $H(x)$ is smooth
, and  considering that its width
is greater than the resolution of the selected grid (which is often
the case for low resolution images), then
$ p(x)<0$ in the area bounded by grid points where $c_j=0$.
  \item
Given $\{x_j~|~c_j > 0\}$ forms a small neighborhood of $\xi$, consisting of $N \geq 1$ points, we can make a local Taylor expansion similar to the one-dimensional case.
First of all we have  that
$p(x_j)=0$ for all such $j$. By summing those with respect to $j$  we get
$$
\sum_{j}  \sum_{k}c_k {H(x_j-x_k)} = \sum_j
{H(x_j-\xi)}-{\lambda N}\; .$$ 
Using the lowest order approximation for small arguments we find
$$N H(0) \sum_{k}c_k = N H(0) - \lambda N + {\cal O}(h^2)\;, $$
i.e. to first order
\begin{equation}\label{eq:approximationCoefficients}
 \sum_{k}c_k = 1 - \frac{\lambda}{H(0)}
\end{equation}
Using this approximation
the Hessian can be computed to leading order as
 \begin{align}
\nabla \nabla p(\xi) = &
\frac{1}{\lambda}H(0)-\frac{1}{\lambda}\sum_j c_j  H(\xi-x_j)-1\\
&= \nabla \nabla H(0)\;,
\end{align}
which is negative definite due to our assumptions on $H$. Hence, $p$ is concave in a
neighbourhood of $\xi$, which implies that its level sets are convex.
The points $x_j$ with $p(x_j)=0$ are on the level set $\{ p = 0\}$, i.e. a convex set around $\xi$. Since
$p(x_k) > 0$ is impossible, there is no other grid point inside the convex hull of the $\{x_j\}$.
\end{itemize}

Thus, from those arguments we see that the active grid points ($c_j > 0$) are to be expected in the convex
hull of $\xi$ on the grid. This can be made rigorous under the assumption that the local grid size around $\xi$ is small and there are no active grid points at large distance from $\xi$, which is confirmed in all our numerical experiments.
Figure~\ref{fig:OptCurve2D} illustrates this behaviour by showing the shape of  the function $p$
and its relationship to the nonzero coefficients of the reconstructed
signal $\mu^N$. 
In this figure, the
small blue dots depict the computational grid, the big blue circles
show the grid points with nonzero entries (i.e.
estimated peaks). For the computations,
the regularization parameter was
set $\lambda=0.1 \lambda_\mathrm{max}$ and $H(x)$ was
 Gaussian with standard deviation $\sigma=1.5\sqrt{2}h$ (where $h$
was the grid resolution).
\begin{figure}[h!]
    \begin{center}
       \includegraphics[width=0.64\textwidth]{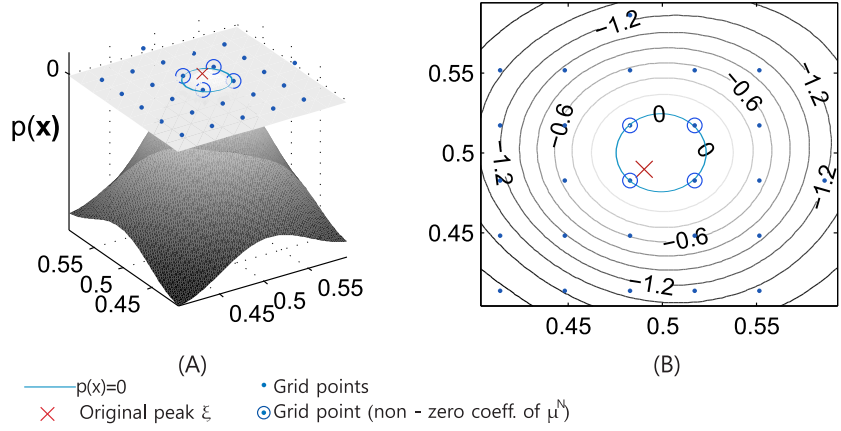}
 \caption{Function $p(x)$ around location $\xi$ when the reconstructed signal $\mu^N$ has four nonzero coefficients. (A)
3D plot of $p(x)$ and xy-plane with the computational grid (marked blue dots) and location of the nonzero coefficients (marked with blue circles)
(B) Isocontours of $p(x)$, computational grid (marked with blue dots) and nonzero locations of $\mu^N$ (marked with blue circles). The exact location denoted by red $\mathrm{x}$.
}
       \label{fig:OptCurve2D}
    \end{center}
\end{figure}
Based on the previous analysis, we can see in  Figure~\ref{fig:lambdaChanging} that the numerical
solution depends on $\lambda$ and the properties of
kernel $H$.
As expected, the number of active grid points increases as $\lambda$ decreases which is
effectively a property of the finite-dimensional $\ell_1$-norm regularization in the convex
hull on the grid.
\begin{figure}[!h]
    \begin{center}
       \includegraphics[width=0.66\textwidth]{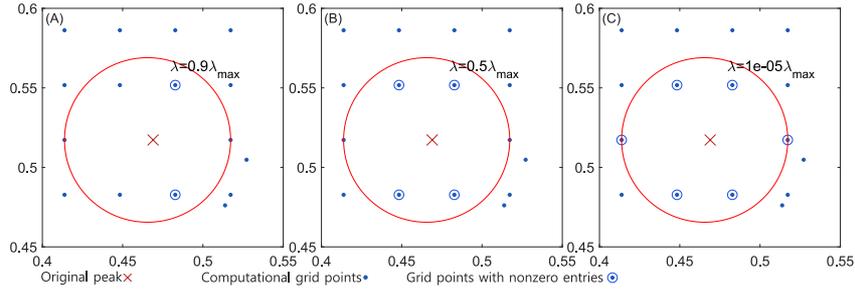}
       \end{center}
       \caption{Numerical solution on a grid  for decreasing value of the regularization parameter
       $\lambda$.
       The blue circles show the locations of the nonzero coefficients of the numerical solution. The blue dots are the grid points. The red circle is the smallest circle that
       encloses the largest convex hull formed by grid points that surround $\xi$ (points that can get nonzero entries).
       The maximum number of nonzero entries is depicted in (C). }
\label{fig:lambdaChanging}
\end{figure}

\section{Recoveries of single peaks using the
$\ell_1$-norm optimality condition}
Let the observations $f$ be of the form $f(x)=\gamma G(x-\xi) +
\tilde f$, with $\tilde f$ supported in some distance to $\xi$. Then, we expect that
problem \eqref{eq:deconvolutionProblem} will yield few nonzero
coefficients $\{c_k\}$ only in a neighborhood of grid
points, ${\cal N}=\{x_k\}$, close to $\xi$ plus additional
non-zeros related to $\tilde f$ in a certain distance.
Hence,
 the confined variational problem is
$$ \min_{(c_k)_{x_k \in {\cal N}}}\| \sum_{x_k \in {\cal N}} c_k G\ast\delta_{x_k} - f \|^2 + \lambda \sum_{x_k \in {\cal N}} \vert c_k \vert\;,$$
where all the $c_k$ have same sign $s \in \{+1,-1\}$.

The associated normal equations are
\begin{equation}
\sum_{x_k \in {\cal N}} c_k G\ast G(x_i-x_k) -G\ast f(x_i) + \lambda
s= 0\;,
\end{equation}
for $x_i \in {\cal N}$ (note that $s$ is the same for all $x_i$).
Given
  $  H(x) = G\ast G(x)=\int_{\Omega} G(x-y)G(y)~dy$,
we can write the associated normal equations as
\begin{equation}\label{eq:associated normal equation}
    \sum_{x_k \in {\cal N}} c_k H(x_i-x_k) - \gamma H(x_i-\xi) + \lambda s= (G*\tilde
f)(x_i) \mbox{ for }\;x_i \in {\cal N}.
\end{equation}
For a small neighbourhood around $\xi$ (and  $h\rightarrow 0$ a
bound for the grid size) we can perform a Taylor-expansion and
obtain
\begin{equation}
\begin{split}\label{eq:OptCond_TaylorExpansion}
    &\left(\sum_{x_k \in {\cal N}} c_k - \gamma\right) H(0)   +  \frac{1}2 \sum_{x_k \in {\cal N}} c_k  (x_i -x_k)^T \nabla^2  H(0)
    (x_i -x_k) - \frac{1}2 \gamma (x_i -\xi)^T \nabla^2  H(0)  (x_i -\xi)
    + \lambda s \\
    & =(G*\tilde f)(\xi) + \nabla  (G*\tilde f)(\xi) (x_i -\xi) + \frac{1}2 (x_i -\xi)^T \nabla^2  (G*\tilde f)(\xi)  (x_i
-\xi) +{\cal O}(h^3)\;,
\end{split}
\end{equation}
 where we have used
$\nabla H(0)=0$. We observe that all equations have the same leading
order term, which yields up to order two
\begin{equation}\label{eq:PeakAmplitude}
    \gamma =    \sum_{x_k \in {\cal N}} c_k +  \frac{\lambda s}{H(0)} -\frac{(G*\tilde f)(\xi)}{H(0)}\;.
\end{equation}
In order to access higher-order terms we can exploit the fact that
set ${\cal N}$ of nonzero coefficients has more than one grid point and thus we can estimate
differences of  equation \eqref{eq:OptCond_TaylorExpansion} for
pairs of grid points $x_i,x_j \in {\cal N}$. This yields
\begin{equation}\label{eq:PeakLoc}
\begin{split}
  \gamma  (x_i -x_j)^T \nabla^2  H(0) \xi
    =& \frac{\gamma}2  \left( x_i^T \nabla^2  H(0)   x_i -x_j^T \nabla^2  H(0)   x_j\right)
\\&-
    \frac{1}2 \sum_{x_k \in {\cal N}} c_k \left( F_k(x_i) - F_k(x_j) \right)\\ & +
\frac{1}{2}\tilde{F}(\xi)+
    {\cal O}(h^3)
\end{split}
\end{equation}
where $F_k(x)= (x -x_k)^T \nabla^2  H(0)  (x -x_k)$ and $
\tilde{F}(\xi)=  (x_i -\xi)^T \nabla^2 (G*\tilde
f)(\xi)(x_i -\xi) - (x_j -\xi)^T \nabla^2  (G*\tilde f)(\xi)  (x_j
-\xi)+ \nabla  (G*\tilde f)(\xi) (x_i - x_j)$.
 Equation \eqref{eq:PeakLoc} can be interpreted as a linear equation for
$\xi\in\Omega$. Having at least $m>d$ different grid points in ${\cal
N}$, we can derive  $m(m-1)/2$ equations.
.
 If we can choose the $x_i - x_j$ to be a
basis of $\R^d$, the negative definiteness of $\nabla^2 H(0)$ and
$\gamma \neq 0$ imply that the matrix formed out of the vectors
$\gamma (x_i -x_j)^T \nabla^2  H(0)$ has  rank $d$.
Thus, we can
uniquely solve for the location $\xi$ and obtain a second order
approximation in $h$ (considering the contribution from $\tilde{f}$
negligible).

\subsection{Examples of peak recoveries}
To demonstrate the previous theoretical results, we present some
examples in one and two dimensional spaces.
\subsubsection{1D spaces}
In the following examples, we consider a signal with three peaks
with amplitudes $\gamma_1$, $\gamma_2$ and $\gamma_3$ at positions
$\xi_1$, $\xi_2$ and $\xi_3$ where $\xi_l\in(0,1)$ for $l=1,2,3$.
The signal is given by
$$ {\mu} = \gamma_1\delta_{\xi_1} + \gamma_2\delta_{\xi_2} +
\gamma_3\delta_{\xi_3} \; . $$ Moreover, we choose a Gaussian
convolution kernel ${G}$ with standard deviation $\sigma = 0.03$.
The continuous convolved data can be expressed analytically as
$$ f(x) = \gamma_1 {G}(x-\xi_1) + \gamma_2
{G}(x-\xi_2) + \gamma_3 {G}(x-\xi_3)\;. $$

For the estimation of the numerical solution $\mu^N$ the domain
$[0,1]$ is discretized and
 the $\ell_1$-norm minimization problem \eqref{eq:deconvolutionProblem} is solved with
$\lambda=0.01\|A^*f\|_\infty$ on a uniform grid of size $N$.

We first consider a grid that
includes $\xi_1$, $\xi_2$ and $\xi_3$. In Figure~\ref{fig:ExactSupport}, we observe that the exact recovery is feasible
(which is in accordance with
preposition \eqref{thm:recon_spikes_on_gridpoints}).
\begin{figure}[!htb]
    \begin{center}
       \includegraphics[width=0.66\textwidth]{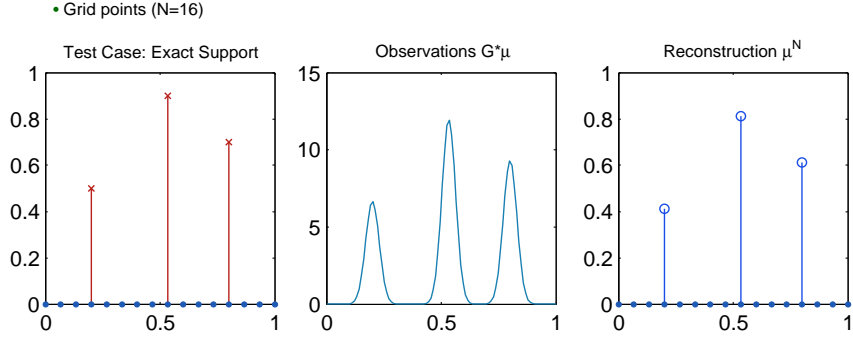}
       \caption{Left image: Original peaks, middle image: observations and right image: solution of the discrete convex problem
       \eqref{eq:deconvolutionProblem}.
       The solution of the $\ell_1$-norm minimization problem for a computational grid
       which includes points $\xi_1$, $\xi_2$ and $\xi_3$ enables the
       exact recovery of the peak positions. The peaks
of  ${\mu}$ are at locations $\xi_1=0.2$, $\xi_2
= 0.533$ and $\xi_3 = 0.8$ and their amplitudes are $\gamma_1 =
0.5$, $\gamma_2 = 0.9$ and $\gamma_3 = 0.7$ respectively. The number
of grid points used here was $N=16$ ( $h = 0.1$).}
       \label{fig:ExactSupport}
    \end{center}
\end{figure}

Now we can consider the case of Theorem
\ref{thm:spikes_between_grid_points} where the three peaks of
${\mu}$ are located between the grid points.
Figure~\ref{fig:NonExactSupport} and~\ref{fig:NonExactSupport_2}
depict the results for two different grids of size $N=16$ and $N=51$
respectively. The numerical solutions yield to either
two peaks around the location of an original peak or a single peak
close to the original one as one expects. 

To approximate the amplitude and location of the underlying peaks we
used equation~\eqref{eq:PeakAmplitude} and \eqref{eq:PeakLoc} respectively.  In
particular, for the approximation of a peak located at $\xi_l \in
(x_{k},x_{k+1})$ with amplitude $\gamma_l$, if there exist two nonzero coefficients $c_k$ and $c_{k+1}$
 at  points $x_k$ and $x_{k+1}$
respectively, then from~\eqref{eq:PeakAmplitude} follows that
\begin{equation}\label{eq:amplitude_1D}
  \hat{\gamma}_l = c_k+c_{k+1}+\frac{\lambda\;s}{H(0)}\;,
\end{equation}
If $c_k>0$ and $c_{k+1}>0$,  $s=1$.

 Based on
equation~\eqref{eq:PeakLoc},  the peak location is approximated as
\begin{equation} \label{eq:PeakLoc_1D}
   \hat{\xi}_l = \frac{1}{2} (x_k+x_{k+1}) +\frac{c_{k+1}-c_k}{2\hat{\gamma}_l} (x_{k+1}-x_k)\;.
\end{equation}

Terms that include $(G\ast\tilde{f})$ in
equation~\eqref{eq:PeakAmplitude} and \eqref{eq:PeakLoc} has been eliminated  from~\eqref{eq:amplitude_1D} and
~\eqref{eq:PeakLoc_1D} since we use only the neighboring point contributions to recover the amplitude and position of the underlying peaks.

We note that  in the case where the numerical solution yields to
a single nonzero coefficient  $c_k$ at $x_k$, then
$\hat{\gamma}_l=c_k$ and $\hat{\xi}_l = x_k$.
\begin{figure}[!htb]
    \begin{center}
       \includegraphics[width=0.66\textwidth]{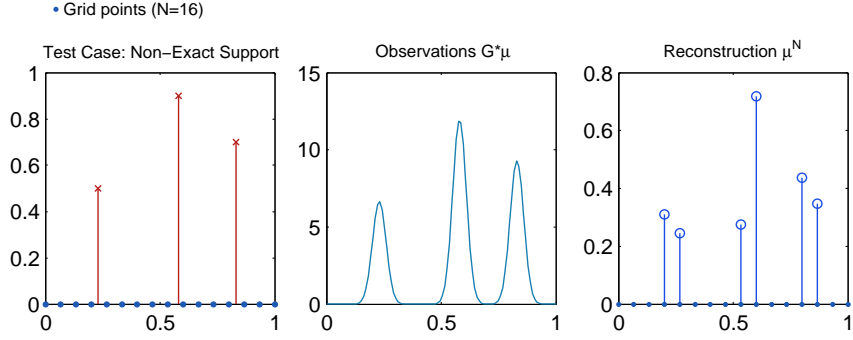}
       \caption{Left image: Original peaks, middle image: observations and right image: solution of the discrete convex problem
       \eqref{eq:deconvolutionProblem}.
       The solution of the $\ell_1$-norm minimization problem for a computation grid
       which does not include points $\xi_1$, $\xi_2$ and $\xi_3$ gives, as a solution, pairs of peaks adjacent to location of the original peak.}
       \label{fig:NonExactSupport}
    \end{center}
\end{figure}
\begin{figure}[!htb]
    \begin{center}
       \includegraphics[width=0.66\textwidth]{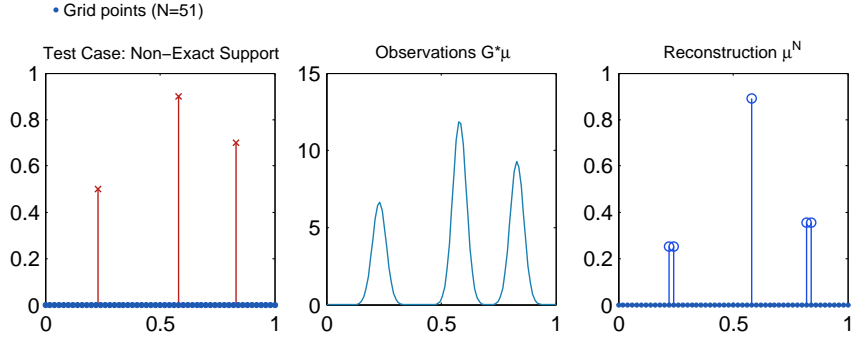}
       \caption{Similar as Figure~\ref{fig:NonExactSupport}, however finer grid is used in the computation.
       The solution yields two pairs of peaks around $\xi_1$ and  $\xi_3$ and a single peak close to $\xi_2$ }
       \label{fig:NonExactSupport_2}
    \end{center}
\end{figure}
Table~\ref{table:NonExactSupport} summarizes the values of the
amplitudes and locations of the original and estimated peaks for the
two different computational grids of Figure~\ref{fig:NonExactSupport} and \ref{fig:NonExactSupport_2}. Based on these results,
the reconstructions in a fine grid are slightly more accurate than ones  obtained using a coarse grid, which is inline with the a-posteriori error analysis presented in section \ref{section error
estimation}.
\begin{table}[!htb]
\centering
\begin{tabular}[c]{*7{m{8mm}}}
     \multicolumn{2}{c }{True}& \multicolumn{2}{c }{N=16}& \multicolumn{2}{c }{N=51}
    \\
    \hline \multicolumn{1}{c} {$\xi$}&\multicolumn{1}{c} {$\gamma$}& \multicolumn{1}{c} {$\hat{\xi}$}&\multicolumn{1}{c} {$\hat{\gamma}$}  &\multicolumn{1}{c}
    {$\hat{\xi}$} & \multicolumn{1}{c} {$\hat{\gamma}$}
\Tstrut \Bstrut
    \\
\hline
    \multicolumn{1}{c} {0.23}& \multicolumn{1}{c} {0.5} & \multicolumn{1}{c}{0.2312 }& 0.55&\multicolumn{1}{c}
    {0.2301}& 0.50
    \\
    \multicolumn{1}{c} {0.58}&\multicolumn{1}{c} {0.9} &\multicolumn{1}{c} {0.5751} & 0.92 & \multicolumn{1}{c}
    {0.5800}& 0.91
    \\
    \multicolumn{1}{c}{0.83}&\multicolumn{1}{c} {0.7}  &\multicolumn{1}{c}{0.8312}& 0.77   & \multicolumn{1}{c}{0.8300} & 0.71\\
    \hline
\end{tabular}
\caption{This table summarizes the results of the test cases
illustrated in Fig.~\ref{fig:NonExactSupport} and
Fig~\ref{fig:NonExactSupport_2}. The first and second column show
the locations $\xi$ and amplitudes $\gamma$ of the underlying peaks,
then there are the estimated locations and amplitudes for the cases
where a coarse grid (N=16 points) and a fine grid (N=51 points) were
used.} \label{table:NonExactSupport}
\end{table}

\subsubsection{2D spaces}
In higher dimensions, the  exact signal is $\mu =
\sum_{l=1}^L\gamma_l\delta_{\xi_l}$ where $\xi_l\in\Omega\subset
\mathbb{R}^d$ ($d>1$) and the estimated $\mu^N$ solution has
nonzero values clustered in grid points around the locations of
the actual peaks $\xi_l$. If a cluster of grid points  with
nonzero coefficients around peak $\xi_l$ is denoted by
$\mathcal{N}_l=\{x_{k}\}_{1:N_l}$, we approximate the peak
amplitude according to
\begin{equation}\label{eq:Peak_AmplitudeND}
\hat{\gamma}_l = \sum_{x_k\in\mathcal{N}_l} c_k+\frac{\lambda \;
s}{H(0)}\;.
\end{equation}
The location is approximated similarly as in
\eqref{eq:PeakLoc}. Particularly, if there are at least two grid
points $x_i$ and $x_j$ $\in\mathcal{N}_l$,  we have
\begin{equation}\label{eq:PeakLoc_ND}
\begin{split}
  \hat{\gamma}_l  (x_i -x_j)^\mathrm{T} \nabla^2  H(0) \hat{\xi}_l
    =& \frac{\hat{\gamma}_l}2  \left( x_i^\mathrm{T} \nabla^2  H(0)   x_i -x_j^\mathrm{T} \nabla^2  H(0)   x_j\right)  \\-&
\frac{1}2 \sum_{x_k \in {\cal N}_l} c_k  \left(F(x_i)-
F(x_j)\right),
   \end{split}
\end{equation}
where $F(x)=(x -x_k)^\mathrm{T} \nabla^2  H(0)  (x -x_k)$.
 Having ${N}_l>d$ different grid points in
$\mathcal{N}_l$, we can approximate the location of $\xi_l$ solving
a set of equations
(\ref{eq:PeakLoc_ND}) which are constructed by selecting one $x_i$ at a time and taking differences to all other $x_j\in\mathcal{N}_l$.

As an example here we have a low resolution image produced as the
convolution of four peaks with a Gaussian kernel in a two
dimensional space. The domain is $\Omega=[0,1]^2$ (left image of
Figure~\ref{fig:2D_example_PeakEstimation}) and the selected
computational grid is of size $N=20\times 20$ (depicted as small
blue dots in the middle and right images of
Figure~\ref{fig:2D_example_PeakEstimation}). The middle image of
Figure~\ref{fig:2D_example_PeakEstimation} shows the numerical
result obtained solving the $\ell_1$-norm minimization problem. The
intense blue circles illustrate the locations where nonzero entries
appeared. We can observe that there are four distinctive clusters of
grid points with nonzero coefficients.

Therefore, four
peaks are approximated, one for each cluster using
\eqref{eq:PeakLoc_ND}. The peak approximations are shown in the
right image of Figure~\ref{fig:2D_example_PeakEstimation}.
\begin{figure}[]
    \begin{center}
       \includegraphics[width=0.65\textwidth]{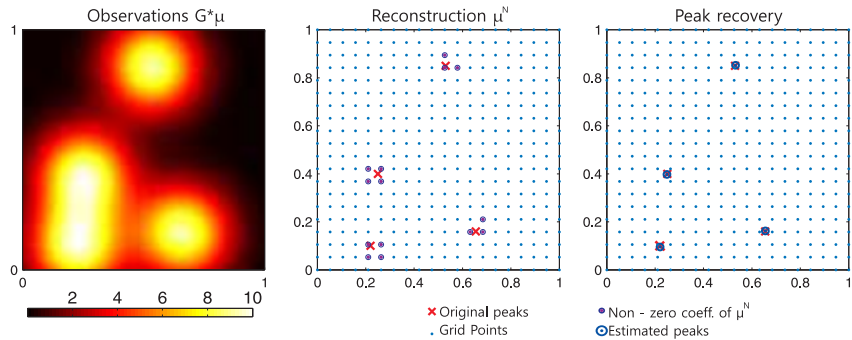}
       \caption{Left image: Observations (convolution with a Gaussian kernel with standard deviation 0.13), middle image: numerical solution $\mu^N$ estimated solving the finite dimensional $\ell_1-$norm problem (with $\lambda = 0.001\lambda_\mathrm{max}$) on a uniform grid $20\times 20$, and right image:
       approximation of the peak locations using the clusters of nonzero coefficients obtained from $\ell_1-$norm minimization.}
        \label{fig:2D_example_PeakEstimation}
    \end{center}
\end{figure}
Table~\ref{table:2D_Example} summarizes the results of the numerical
solution $\mu^N$ and the corresponding approximate peaks $(\hat{\gamma},\hat{\xi})$.

\begin{table}[!htb]
\centering
\begin{tabular}[c]{*4{m{6mm}}}
    \multicolumn{2}{c}{Location}& \multicolumn{2}{c}{Amplitude}
    \\
    \cline{1-4}
    \multicolumn{1}{c} {$\xi$}& \multicolumn{1}{c}{$\hat{\xi}$}  & \multicolumn{1}{c} {$\gamma$} & \multicolumn{1}{c}{$\hat{\gamma}$} \Tstrut\Bstrut \\
    \cline{1-4}
    \multicolumn{1}{c}{(0.22,0.10)} &  \multicolumn{1}{c}{(0.2204,0.0950)}  &\multicolumn{1}{c}{1}  &\multicolumn{1}{c}{0.99} \\
    \multicolumn{1}{c}{(0.66,0.16)}  &  \multicolumn{1}{c}{(0.6557,0.1620)} &\multicolumn{1}{c}{1} &\multicolumn{1}{c}{1.10}   \\
    \multicolumn{1}{c}{(0.53,0.85)}  &  \multicolumn{1}{c}{(0.5323,0.8525)  } &\multicolumn{1}{c}{1}  &\multicolumn{1}{c}{1.02}  \\
    \multicolumn{1}{c}{(0.25,0.40)}  &  \multicolumn{1}{c}{(0.2487,0.3977) } &\multicolumn{1}{c}{1}   &\multicolumn{1}{c}{1.11}      \\
    \hline
\end{tabular}
\caption{ Locations, $\xi,$ and $\hat{\xi}$ and amplitudes, $\gamma$ and  $\hat{\gamma}$ of  the original and estimated peaks  respectively for the case presented in
Figure~\ref{fig:2D_example_PeakEstimation}.}
\label{table:2D_Example}
\end{table}
\section{Recoveries in the case of multiple peak signals}\label{sec:AdaptiveSuperResAlg}
A coarse computational grid, even though reduces the computational
cost, imposes some limitations to detect and separate neighbouring
peaks. For example, there is always a possibility that there are
more than one positive peak in an interval between two grid points
(see one dimensional example of Figure~\ref{fig:MultiplePeaks}.A),
or two or more original peaks may be located in adjacent intervals (e.g.
Figure~\ref{fig:MultiplePeaks}.B). Then the numerical solution
$\mu^N$ of the $\ell_1$-norm minimization
problem may not be accurate enough.

Following similar analysis as in Theorem
\eqref{thm:spikes_between_grid_points},
we can easily show that there exists a $\ell_1$-norm solution
$\mu^N=a\delta_{x_k}$ with
$a=\frac{\sum_{l=1}^{L}\gamma_l H(x_k-\xi_l)-\lambda}{H(0)}$ as depicted in
Figure~\ref{fig:MultiplePeaks}.A even though the original peaks are
two. Additionally, when the original peaks are distributed between
two intervals then we can expect up to three reconstructed peaks as
in Figure~\ref{fig:MultiplePeaks}.B and
Figure~\ref{fig:MultiplePeaks}.C.
\begin{figure}[!htb]
    \begin{center}
       \includegraphics[width=0.76\textwidth]{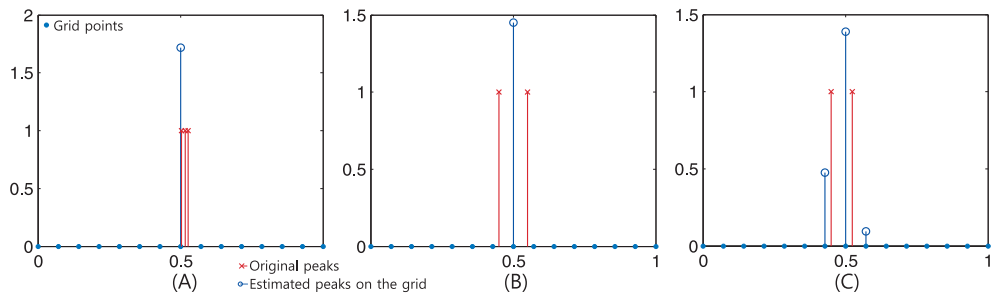}
       \caption{
In this figure, we show the peak recovered when there are more than one original peaks.
       (A) A single peak reconstruction (in blue) when the original peaks (in red) are very close to a grid point $x_k$  (B) A single peak reconstruction (in blue) when the original peaks (in red) are symmetrically located with respect to a
       grid point. (C) Three peak reconstruction (in blue) when there are two original peaks (in red) on adjacent intervals.}
        \label{fig:MultiplePeaks}
    \end{center}
\end{figure}
A natural way to improve the estimates is by refining the grid.
 Figure~\ref{fig:IntroAdaptiveAlgo} illustrates how
by performing local refinements on the grid (and fitting the
input data with a solution in the updated grid), we can achieve a
separation of the
underlying peaks.

\begin{figure}[h!]
    \begin{center}
    \includegraphics[width=1.05\textwidth]{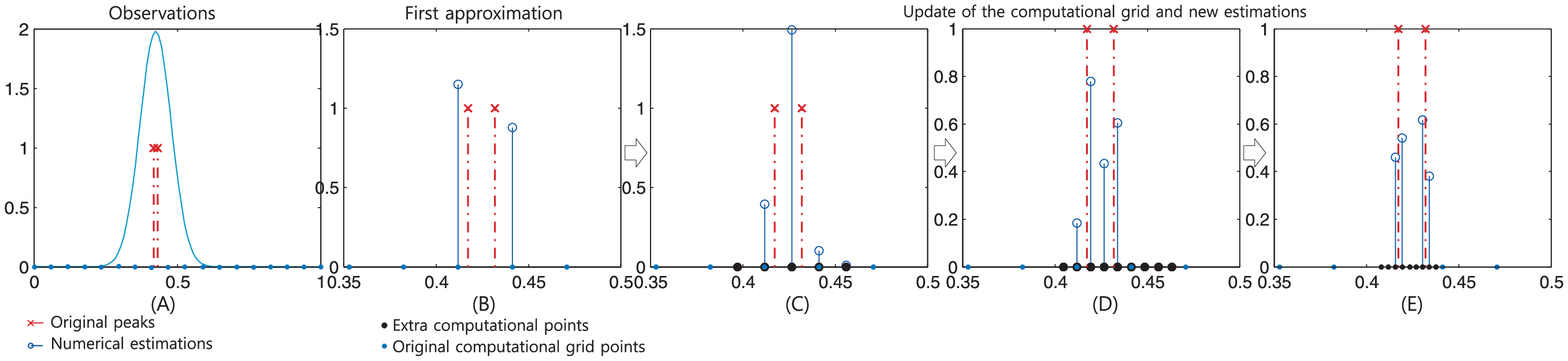}
       \caption{By restricting and refining iteratively the computational grid around the nonzero coefficients of the numerical solution $\mu^N$
       ,
       we can determine small disjoint intervals where the two original peaks (in red) belong to and thus separate them. Fig.(A) shows the observations, original peaks (in red) and the locations of the used computational grid (small blue dots).
 Fig.(B)-(E) illustrate the reconstructed peaks (in blue) on a grid that is updated based on the previous numerical solution.}
       \label{fig:IntroAdaptiveAlgo}
    \end{center}
\end{figure}

\subsection{Adaptive super-resolution for sparse signal}
\subsubsection{Overview}
In the following context, to ease our analysis and  to proceed with the domain refinement in higher dimensions instead of using grid/points
we use the mesh/nodes notion as in the finite element methods.
Hence, the computational domain can be described by a mesh consisting of a
set of nodes (equivalent to grid points) and elements (e.g. line
segments in one dimension or triangles in
two dimensions). 
The proposed super-resolution approach consisting of the following
steps:
\begin{description}
  \item[a)] Solve the $\ell_1$-norm minimization
problem (\ref{eq:deconvolutionProblem}) on a set of given nodes
using a non-smooth convex solver;
  \item [b)] Define a new (restricted) computational domain using the
 nodes (locations) corresponding to the nonzero coefficients of the
estimated numerical solution (\ref{eq:numericalSolution}). To do
that:
\begin{enumerate}
        \item Remove the elements where all their nodes are assigned to zero coefficients;
        \\Cluster all the remaining elements. A cluster is defined by a set of pairwise connected elements (elements that share an edge or surface);
        \item  Refine the computational domain at the estimated
        clusters by including extra nodes (i.e. centroids of the elements).
        Include only the extra nodes that satisfy a distance limit from the existing nodes.
        \item  For each cluster, use the old and  additional nodes to produce a mesh.
        The new (fragmented) computational  domain consisting of all the  disjoint clusters;
\end{enumerate}
\item [c)]
Repeat step a-b until domain stops updating i.e.
 the distance between the existing and additional nodes
becomes sufficiently small\footnote{For example, determine a minimum distance for the nodes using limits presented in
\cite{Poon2018,Poon2019} or prior information about the size of the compressible signal.};
  \item [d)]Based on the last numerical solution (step c),
for each separate cluster, estimate a single peak i.e. amplitude
\eqref{eq:Peak_AmplitudeND} and location using the coordinates of
the nodes corresponding to the nonzero coefficients of the numerical
solution and the set of equations stemming from equation
(\ref{eq:PeakLoc_ND}).
\end{description}
A more analytical description of the approach is given in
Appendix~\ref{sec:Implementation}. We need to mention that the proposed scheme is applicable for
peaks of similar sign or when the positive and negative peaks
satisfy the separation criteria (which is based on the kernel's
 width and noise type/level) as studied in some cases  for example in
\cite{Duval2015b,Poon2018}.

\section{Results and Discussion}\label{sec:results}
To demonstrate how the proposed  superresolution approch can be used, we
reconstruct super-resolved images from (low resolution) observations
which are the convolution of an original sequence of sparse Dirac
delta functions with  Gaussian kernels. In this section, we present
technical details about the simulated data, the proposed
super-resolution approach and the validation metrics used for the
comparison between the original peaks and the estimated ones. Then,
we show examples  how the proposed scheme progressively
localizes a different number of peaks, which can be either only
positive or positive and negative. Finally, we  discuss further extensions and possible applications.

\subsection{Simulated data}
The simulations were carried out in a two dimensional square domain
$\Omega =[0,1]^2$. The aim was to approximate the locations and
amplitudes of an original signal $\mu=\sum_{l=1}^L\gamma_l
\delta_{\xi_l}$ from low resolution images $W\in\mathbb{R}^{M\times
M}$ where
$$w_{j_1j_2} =\sum_{l=1}^L%
\gamma_l  G(\xi_l-x_{j_1j_2})+\varepsilon_{j_1j_2}\;,$$ for
$j_1,j_2=1,\ldots,M$ (where $\xi_l\neq x_{j_1j_2}\forall j_1,j_2$).

As a convolution kernel, we use the one from study \cite{Falcon},  given by
 $$G(x)\propto\alpha
\exp{\left(\frac{1}{2}\left(x^\mathrm{T}\Gamma_1^{-1}x\right)\right)}+(1-\alpha)\exp{\left(\frac{1}{2}\left(x^\mathrm{T}\Gamma_2^{-1}x\right)\right)}\;,$$ 
with $\alpha=0.2$,
 covariance matrices $\Gamma_1=\sigma_i I^{2\times 2}$,
$\Gamma_2=\sigma_2 I^{2\times 2}$,
$\sigma_1=2\;h_M$, $\sigma_2=2.5h_M$ and $h_M=1/{M}$.

Also, we considered a low additive
measurement noise
$\varepsilon=\mathrm{sc}\;\bar{\varepsilon}\in\mathbb{R}^{M\times
M}$ where
$\bar{\varepsilon}\in\mathbb{R}^{M\times M}$ was sampled from a
Gaussian distribution with zero mean and variance one. The scaling
parameter $\mathrm{sc}$ was estimated based on the level of the
signal-to-noise ratio (SNR), $\mathrm{SNR}=
10\log_{10}\frac{\sum_{j_1,j_2=1}^M(\sum_{l=1}^L
\gamma_lG(\xi_l-x_{j_1j_2}))^2}{\sum_{j_1,j_2=1}^M(\varepsilon_{j_1j_2})^2}$.
In the following simulations, we used $\mathrm{SNR}=40\;\mathrm{dB}$.

\subsection{Details about the adaptive superresolution approach}
We estimated the locations and the amplitudes of the underlying
peaks by employing the proposed scheme of
section~\ref{sec:AdaptiveSuperResAlg}.
The initial estimation (by solving the $\ell_1$-norm minimization) was performed in a uniform mesh
of $N\times N$ nodes. Then, the mesh was
updated automatically around the nonzero entries of vector $c$. In
practice, to avoid small numerical inaccuracies, the new domain was
defined by keeping the nodes  with 
absolute values of the
estimated peaks  greater than a small threshold (i.e. 0.5\% of
the maximum $|c|$ of vector $c$).
In the current implementations, the $\ell_1$-norm minimization problem
was solved using the
hierarchical adaptive lasso (HAL) \cite{Murphy2012}.
Other algorithms e.g.\cite{Boyd2017,Kim2007} could be used as well. Here, we used HAL to  reduce the amplitude shrinkage of  the estimated nonzero coefficient given a regularization parameter $\lambda$. In the following examples, $\lambda = 0.1\lambda_\mathrm{max}$.
Moreover, the incorporation of a Bregman iteration \cite{Yin08bregmaniterative} could be considered in the future for  cases with relative high measurement noise.

The updates of the computational support terminated when the distance between the existing nodes and the additional nodes
 became small.
  In the following examples, we used as a criterion for adding  a new node, the minimum distance of this candidate node from the existing nodes, $h_\mathrm{min}=0.25 h_M$ (approximately $0.125$ of the Gaussian kernel's standard deviation).
This choice was made to enable a computational efficiency (i.e. a reasonable number of iterations) and to allow a good approximation of the peaks using   small clusters of the nonzero coefficients and (\ref{eq:PeakLoc_ND}).

\subsection{Comparison metrics}
In tests with only few peaks, we used:
\begin{itemize}
  \item The mean localization error (MLE)
  between the original and reconstructed peaks which is defined
\begin{equation}\label{eq:MLE}
\mathrm{MLE} = \frac{1}{{L}} \sum_{{l}=1}^{\hat{L}} \min_d{\left(
d(\xi_l,\hat{\xi}_{\hat{l}})\right)_{\hat{l}=1:\hat{L}}},\end{equation}
\noindent where $d(\xi,\hat{\xi}_{\hat{l}})
=\|\xi-\hat{\xi}_{\hat{l}}\|_2$,
 $\hat{L}$ is the number of the reconstructed peaks $\hat{\xi}_{\hat{l}}$ and $L$ the total number of the original peaks.
  \item The mean strength error (MSE) given by
\begin{equation}\label{eq:MSE}
\mathrm{MSE} = \frac{1}{L} \sum_{l=1}^{L}
\|\gamma_l-\hat{\gamma}_{\bar{l}}\|\,\,\mbox{where }
\bar{l}:=\min_d{
{\left(d(\xi_l,\hat{\xi}_{\hat{l}})\right)_{\hat{l}=1:\hat{L}}}}.
\end{equation}
\end{itemize}
For  dense distributions of peaks, we employed the
 earth mover's distance (EMD) (Wasserstein metric) as a measure of dissimilarity between the original and the
 estimated peaks locations
\cite{Rubner2000,pel09}.

\subsection{Examples}
Three different examples are presented to demonstrate the different
stages of the proposed scheme. In the first example, we show step
by step the estimation of the locations and amplitudes of five
positive peaks (see Figure~\ref{fig:5Source_noNoise}). In the second
example, presented in Figure~\ref{fig:4Sources_noNoise_Pos_Neg}, we
use the proposed apprach to recover both positive and negative
peaks. In the last example in Figure~\ref{fig:tubuli}, we illustrate
the potential of the algorithm to deal with denser peak
distributions.
\begin{figure*}[h!]
    \begin{center}
       \includegraphics[width=0.84\textwidth]{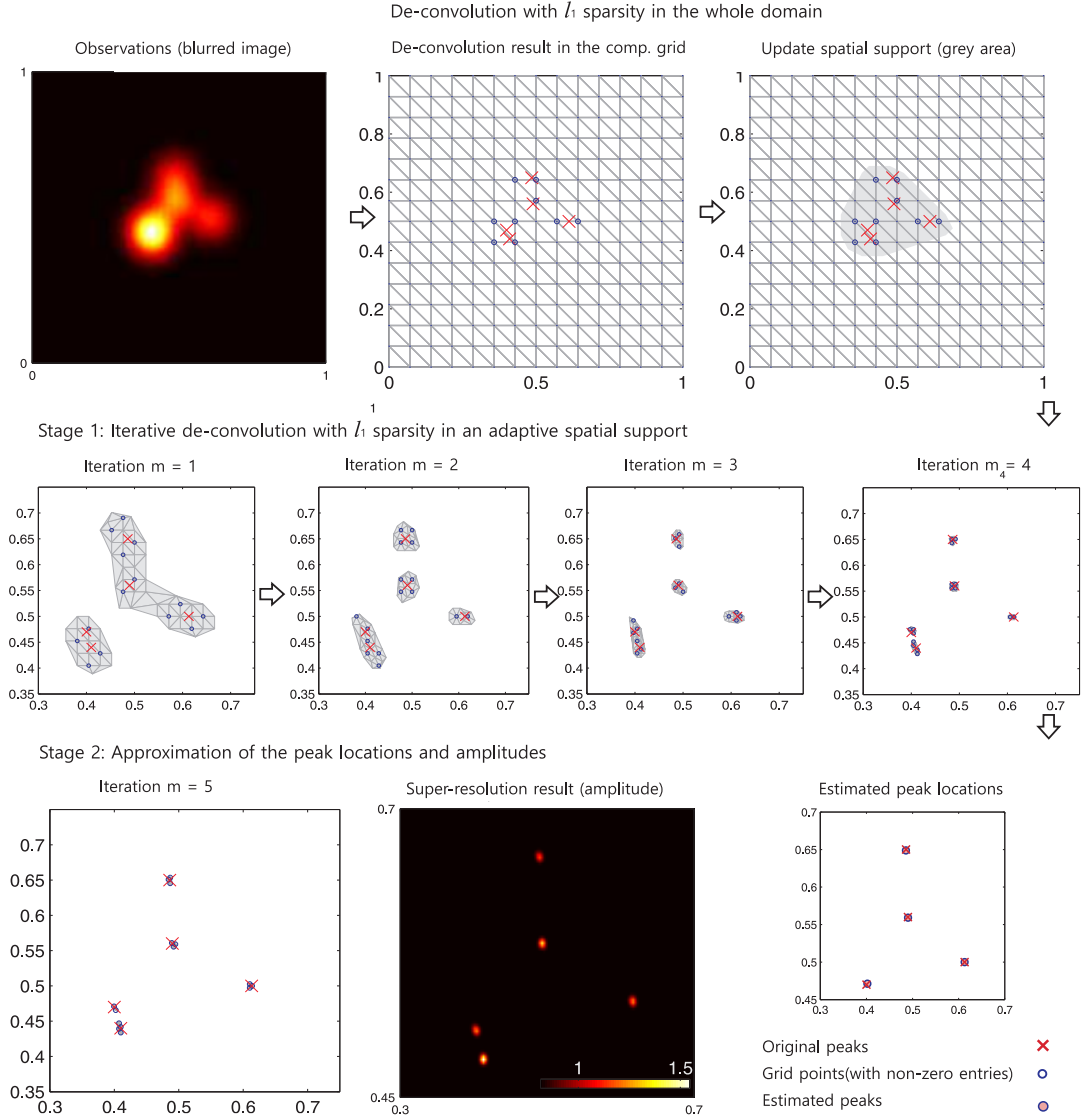}
       \caption{Reconstruction of five positive peaks using the
       proposed adaptive super-resolution algorithm in a $[0,1]^2$ domain.  Top row: the left
image shows the observations on a $[40\times 40]$
       grid, the middle image shows the numerical solution of iteration $m=0$ obtained solving the finite dimensional $\ell_1$-norm
       minimization problem on a computational $[15\times 15]$ domain. The right image depicts  the new computational support in light grey color.
       Middle row: the left image shows the
       numerical solution of  iteration $m=1$ and the new computational support in light grey. Accordingly, the middle and right images show the
       numerical solution and the computational support of iteration $m=2$ and $m=3$ respectively.
       Bottom row: the left image shows the result of the last iteration, the middle image is the super-resolved  results and  the right image shows the exact and the estimated peaks
       image. Please note the limits of the axes have been updated in the images in order to focus in the area that the point  sources are located.
       }
        \label{fig:5Source_noNoise}
    \end{center}
\end{figure*}

 In the example of Figure~\ref{fig:5Source_noNoise} and~\ref{fig:4Sources_noNoise_Pos_Neg}, the
low resolution images were of size $[{M}\times {M}] = [40\times 40]$.
The first estimation solving the $\ell_1$- norm minimization problem
was performed in a uniform computational mesh (with $[{N}\times {N}]
= [15\times 15]$ number of nodes) as we can observe in the top row,
middle image of Figure~\ref{fig:5Source_noNoise}.
In particular, in Figure~\ref{fig:5Source_noNoise} along the top
row, starting from left to right, we can observe the initial low
resolution image, next the numerical solution of the first iteration (denoted by $m=0$)
 and then, the new
computational domain after the first estimation marked with gray
color. The nodes corresponding to the nonzero entries for vector $c$
on the computational meshes were denoted with small blue circles and
the original peaks were marked with
$\textcolor[rgb]{1.00,0.00,0.00}{+}$.
The new computational area was defined using the
elements where the blue circled nodes belonged to. The second row of
Figure~\ref{fig:5Source_noNoise} illustrates the nonzero locations
estimated by solving the minimization problem
\eqref{eq:mRecursionMinimizationProblem} and the corresponding
updated computational supports for the first, second, third and forth
iteration of the proposed approach. In the last row, we can observe the
$\ell_1$-norm estimation for the last iteration, the
final high resolution image and the approximation of the peaks using
the coordinates of the nonzero coefficients at the last iteration.

The reconstruction results (amplitudes and locations) are summarized
in Table~\ref{table:5Source_noNoise}. Based on them, we have that
all the five peaks were recovered and their values are very close to
the exact values.
%
\begin{table}[h!]
\centering
\begin{tabular}[c]{*4{m{6mm}}}
    \multicolumn{2}{c}{Location}& \multicolumn{2}{c}{Amplitude}
    \\
    \multicolumn{1}{c} {$\xi$}& \multicolumn{1}{c}{$\hat{\xi}$}  & \multicolumn{1}{c} {$\gamma$} & \multicolumn{1}{c}{$\hat{\gamma}$}\\
    \cline{1-4}
    \multicolumn{1}{c}{(0.195,0.58)} &  \multicolumn{1}{c}{(0.198,0.58)}  &\multicolumn{1}{c}{1}  &\multicolumn{1}{c}{0.99} \\
    \multicolumn{1}{c}{(0.18,0.72)}  &  \multicolumn{1}{c}{(0.18,0.72)} &\multicolumn{1}{c}{1.5} &\multicolumn{1}{c}{1.51}   \\
    \multicolumn{1}{c}{(0.48,0.46)}  &  \multicolumn{1}{c}{(0.48,0.45)  } &\multicolumn{1}{c}{1}  &\multicolumn{1}{c}{1}  \\
    \multicolumn{1}{c}{(0.72,0.38)}  &  \multicolumn{1}{c}{(0.71,0.38) } &\multicolumn{1}{c}{1}   &\multicolumn{1}{c}{0.98}      \\
    \multicolumn{1}{c}{(0.64,0.36)}  &  \multicolumn{1}{c}{(0.65,0.37) } &\multicolumn{1}{c}{1.2}   &\multicolumn{1}{c}{1.2}      \\
    \cline{1-4}
    \multicolumn{1}{c} {MLE:}        &   \multicolumn{1}{c}{0.0049}        &\multicolumn{1}{c} {MSE:} &\multicolumn{1}{c}{0.013} \\
    \hline
\end{tabular}
\caption{This table summarizes the results of the test case in
Fig.~\ref{fig:5Source_noNoise}. The first and the third columns show
the locations and the amplitudes of the original peaks respectively
and the second and forth column, the corresponding estimated values
using the proposed approach. The last row shows the values of the
MLE (equation \ref{eq:MLE}) and MSE (equation \ref{eq:MSE})}

\label{table:5Source_noNoise}
\end{table}

%
To show that the proposed scheme can be used to recover both positive and negative peaks,
Figure~\ref{fig:4Sources_noNoise_Pos_Neg} illustrates the estimation
of two positive and two negative peaks.
The original peaks are marked with
$\textcolor[rgb]{1.00,0.00,0.00}{+}$ for the positive peaks and
$\textcolor[rgb]{1.00,0.00,0.00}{\mathrm{x}}$ for the negative
peaks and the estimated ones with circles and squares respectively in the right-hand side image of Figure~\ref{fig:4Sources_noNoise_Pos_Neg}.
Table~\ref{table:PosNegPeaks} summarizes the values of the location
and the amplitude of the original and the estimated peaks of the example in Figure ~\ref{fig:4Sources_noNoise_Pos_Neg}. The
estimated locations and amplitudes are very close to the original
ones with small values for the metrics MLE and MSE for this setup.
\begin{figure}[h!]
  \begin{center}
       \includegraphics[width=0.7\textwidth]{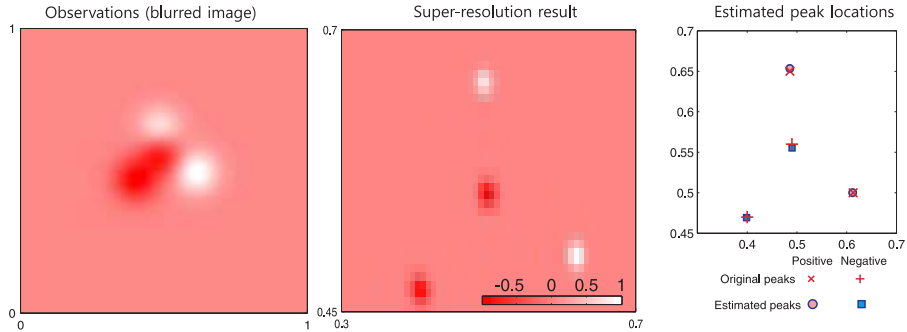}
       \caption{Positive and negative peak reconstructions. Starting from left,
 the original low resolution image, estimated high resolution image and the locations of the actual and estimated peaks.
 Please note the limits of the axes have been updated in the middle and right image in order to focus in the area that the point sources are located
}
        \label{fig:4Sources_noNoise_Pos_Neg}
    \end{center}
\end{figure}
\begin{table}[h!]
\centering
\begin{tabular}[c]{*4{m{6mm}}}
    \multicolumn{2}{c}{Location}& \multicolumn{2}{c}{Amplitude}
    \\
    \multicolumn{1}{c} {$\xi$}& \multicolumn{1}{c}{$\hat{\xi}$}  & \multicolumn{1}{c} {$\gamma$} & \multicolumn{1}{c}{$\hat{\gamma}$}\\
    \cline{1-4}
    \multicolumn{1}{c}{(0.49,0.56)} &  \multicolumn{1}{c}{(0.4857,0.556)}  &\multicolumn{1}{c}{-1}  &\multicolumn{1}{c}{-0.9886} \\
    \multicolumn{1}{c}{(0.486,0.65)}  &  \multicolumn{1}{c}{(0.4857,0.6533)} &\multicolumn{1}{c}{0.8} &\multicolumn{1}{c}{0.7528}   \\
    \multicolumn{1}{c}{(0.4,0.47)}  &  \multicolumn{1}{c}{(0.3989,0.4694)  } &\multicolumn{1}{c}{-1}  &\multicolumn{1}{c}{-0.9739}  \\
    \multicolumn{1}{c}{(0.613,0.5)}  &  \multicolumn{1}{c}{(0.612,0.5) } &\multicolumn{1}{c}{1}   &\multicolumn{1}{c}{1.0076}      \\
    \cline{1-4}
    \multicolumn{1}{c} {MLE:}        &   \multicolumn{1}{c}{0.0025}        &\multicolumn{1}{c} {MSE:} &\multicolumn{1}{c}{0.0231} \\
    \hline
\end{tabular}
\caption{This table summarizes the results of the test case
illustrated in Fig.~\ref{fig:4Sources_noNoise_Pos_Neg}. The first
and the third columns show the locations and the amplitudes of the
original peaks respectively and the second and forth column, their
estimated values using the proposed approach. The last row shows
the values of the MLE (equation \ref{eq:MLE}) and MSE (equation
\ref{eq:MSE})
} \label{table:PosNegPeaks}
\end{table}
\begin{figure*}[h!]
    \begin{center}
       \includegraphics[width=0.76\textwidth]{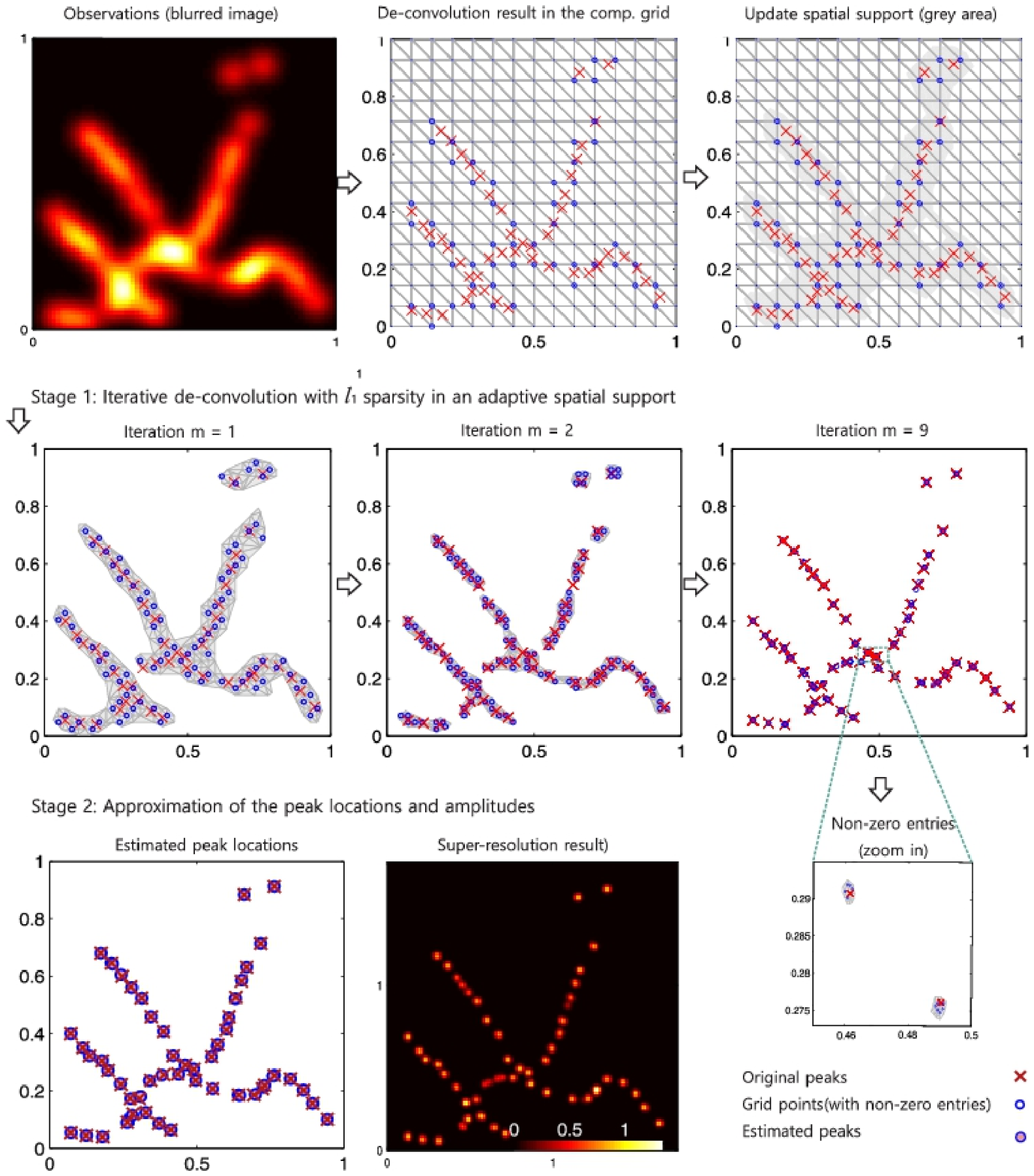}
       \caption{Reconstruction of multiple peaks. The uppermost left images shows the original low resolution observations and similarly as it was described in
       Figure~\ref{fig:5Source_noNoise}, we present the progressive steps of the approach to separate the underlying peaks.  Also, the image on the lowermost left side
shows  the grid points corresponding to non-zero coefficients in the
clusters (in grey)  of the 9$^{th}$ iteration of the approach}
        \label{fig:tubuli}
    \end{center}
\end{figure*}
We further examined the proposed scheme in the case where
50 peaks of intensity one were simulated.  Figure~\ref{fig:tubuli} presents in a
similar way as Figure~\ref{fig:5Source_noNoise} the progressive
towards the recovery of the peaks. In this test, the
low resolution image was $[{M}\times {M}] = [80\times 80]$ and the
first numerical estimation was performed in a uniform
mesh with  $[15\times 15]$ nodes. Here, the
middle row of Figure~\ref{fig:5Source_noNoise} shows the estimates in the first, second and ninth
iteration.
Also, the small lowermost right image illustrates how the numerical
solution (blue circles) appears in a small area around 2 peaks.
The total number of  recovered
peaks was 49. There is an omission due to the very close proximity of two peaks which appear as a single (more intense) peak in the lower right side of the image ``super-resolution result''. Very few of the peak intensities were more prominent than others.
The EMD as a measure of dissimilarity between the actual and
estimated point sources, in domain  $[0,1]^2$, was $\mathrm{EMD} =
0.01$ (or $1\%$ dissimilarity between the true source distribution and the estimated one).

\begin{figure}[h!]
    \begin{center}
       \includegraphics[width=0.55\textwidth]{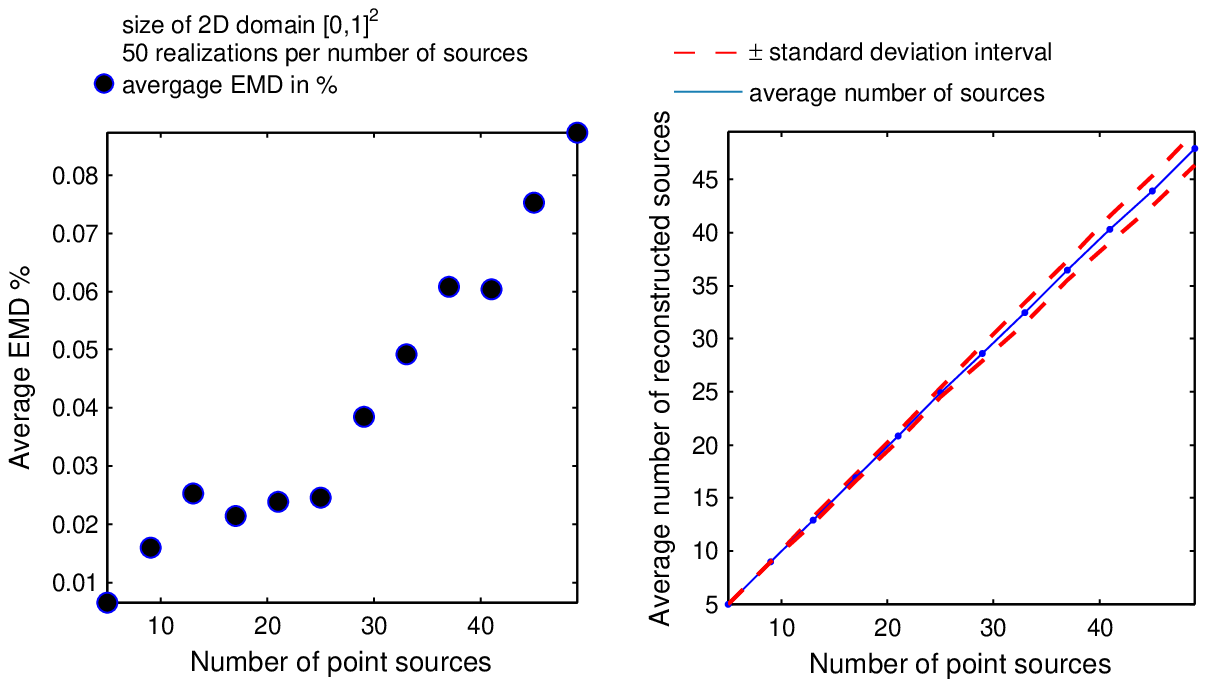}
       \caption{Left image: Average EMD value in \%  and Right image: average number of sources for increasing number of simulated sources.
        50 different realization per number of source were simulated.}
        \label{fig:EMD}
    \end{center}
\end{figure}
\begin{figure}[!h]
    \begin{center}
       \includegraphics[width=0.55\textwidth]{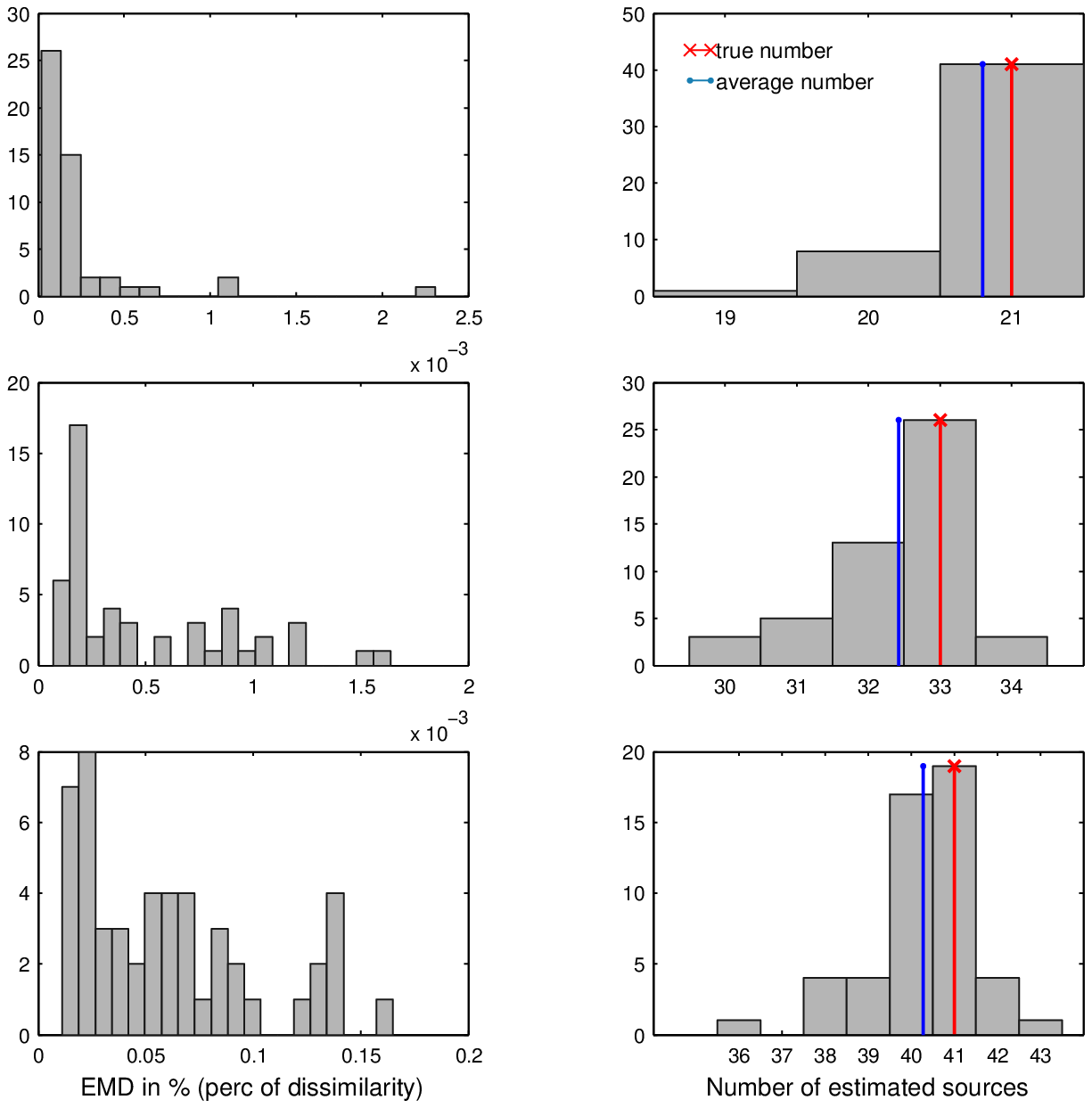}
       \caption{Histograms of estimated EMD values (left columns) and number of reconstructed sources (right column)
when the number of simulated sources was 21, 33, 41 respectively
starting from the top row.}
        \label{fig:Histograms}
    \end{center}
\end{figure}
Finally by keeping the noise level at 40dB, we performed
reconstructions using sets of low resolution images obtained from
the convolution of randomly distributed sources with the same
Gaussian kernel ($\alpha=0.2$). The domain and the properties of the kernel were
the same as in the test of Figure~\ref{fig:tubuli}. Particularly,
given the number of point sources, 50 randomly created source
distributions were generated to produce 50 low resolution images. In
Figure~\ref{fig:EMD} we have the average EMD values in $\%$
estimated by comparing the true point sources with the estimated
ones (left image) and the average number of reconstructed sources
(right image) for increasing number of point sources.

Moreover, based on the histograms of Figure~\ref{fig:Histograms} we
can observe when the  number of point sources is low the localization
error expressed through the EMD value is also low whereas when the number
of point sources increases we have larger variation in the EMD and source number estimates.
This can be explained by the very close proximity of some (true) point sources that
can occur more likely when their number increases in the confined domain $[0,1]^2$. This can lead to  difficulties in separating some of the point source from each other.
Overall, our demonstrations indicate that, in a low noise regime,
the proposed superresolution adaptive scheme can recover as many peaks as the exact number of them in most of the cases
when their distance does not violate
an underlying separation condition e.g. minimum distance
$\Delta_{min}\geq C\sigma$ for Gaussian kernels \cite{Poon2018}. Our
numerical simulations showed that, when the noise level was SNR=40dB, the correct
number of peaks could be recovered if $\Delta_{min}/\sigma>0.2$ given $h_M$.
We remark that even though  the effect of the measurement noise (either Gaussian or
Poisson) and the measurement/observation sampling can affect the
recovery of a multi-dimensional signal, these mathematical questions admit of different analysis than
the currently addressed questions  and will be considered in a follow-up study.

\subsection{Discussion and future prospects}
Questions regarding the deconvolution of sparse peaks present great mathematical difficulties with some of them investigated  in studies  such as in \cite{Duval2013a,Duval2015,DuvalDec.2015,Denoyelle2016,Poon2019}.
However, even though very important, these studies often 
do not accommodate easily accessible solutions to software developers
and engineers
working on superresolution applications.
The current work aimed to shed light into some of these theoretical
findings and put them into perspective with  practical solutions in superresolution algorithms.

In the theoretical part of this study,
we explained why clusters of
nonzero peaks appear around the locations of the original peaks
when we solve the $\ell_1$-norm minimization problem on a discrete
grid and in which parts of the grid these nonzero peaks are more
likely to appear. Moreover, we showed how the locations of the underlying peaks are connected with the numerical solution using the optimality condition of the $\ell_1$- norm minimization problem.
One important remark here is that the
distribution of the nonzero coefficients of the numerical solution depend
on the properties of the convolution kernel $H=G*G$  \eqref{eq:p(x)} and not on
kernel $G$ \eqref{eq:convolution}. Therefore thinking in more
general terms, for an inverse problem with forward operator $R$, it
is only important that $R^*R$ is a convolution to anticipate a
numerical solution following similar pattern as in the current
problem. This could be true for example in classical tomography
e.g. filtered back-projection \cite{Natterer2001}.

In general, we envision that similar superresolution schemes can be performed in a
wide variety of inverse problems in the fields of geophysics,
astronomy and spectroscopy \cite{Kirsch1996,Engl1996,Kaipio2004}
because many of these applications share the same characteristics
and properties as this deconvolution problem.
However, some new aspects need to be investigated. For example, in
neuroimaging, the EEG source imaging problem, even though it shares
seemingly similarities with the current deconvolution,
problem,  is a severely ill-posed problem where the forward
operator has a singularity and its computational version
has a  matrix with a large null space
\cite{Michel2004}. Therefore, special design of the prior model
(e.g. weighting) is required whereas the expected pattern of the
numerical
reconstructions has to be  studied carefully.

In the application part, the main two novelties of the proposed approach were a) the automatic
adaptation of the computational domain (using elements) and b) the approximation of the
underlying peaks using a numerical approximation of the $\ell_1$
norm optimality condition that stemmed from the findings of our
theoretical analysis. As a natural next step though, within the
microscopy field we expect to compare the proposed approach with
other state-of-the art algorithms \cite{Small2014} which lie either
on the variational or spectral framework, e.g. Alternating Descent
Conditional Gradient Method \cite{Boyd2017} or MUSICAL
\cite{Agarwal2016} respectively.

Moreover, we are considering extension that could improve the
algorithmic performance, for example, the incorporation of a
non-convex step as in \cite{Boyd2017} to possibly speed up 
the convergence.
In that step possibly, the optimality curve \eqref{eq:p(x)}
could guide the update of the computational domain simultaneously in
multiple locations.
 Furthermore, the idea of employing the ensemble
learning or committee method \cite{Murphy2012} (which allows to
estimate a weighted solution  in each discretization level by
solving multiple deconvolution problems in a parallel fashion),
could help to reduce possible bias introduced due to the
regularization or high measurement noise.

\section{Conclusions}\label{sec:conclusion}
The current work bridges the gap
between theoretical studies and implementations of algorithms that impose sparsity constraints on
the signal to be recovered.
First,
we studied theoretically the deconvolution of single peaks using
the $\ell_1$- norm and we confirmed recent observations that a
discrete reconstruction yields to multiple peaks at grid points
adjacent to the location of the actual peak. We showed that by using
these adjacent peaks and the first order optimality condition of
this convex problem, we can obtain a set of linear equations to
approximate the location of the
actual peak.
We quantified the errors between  the continuous (TV) problem (that allows
exact peak recoveries) and the finite $\ell_1$- norm minimization
problem, which designated that the accuracy of the numerical estimates depends on the
discretization that can be improved by applying finer gridding.

Second, using the previous theoretical finding we proposed an
iterative scheme in which automated local refinement on the
computational grid was performed to identify the areas where the
true peaks were located. Then, with the help of the equations from
the optimality condition,  the peak locations and amplitudes were estimated.
Finally, low resolution images, obtained using simulated focal
sources convoluted with a smooth kernel, were used to show that our
approach can increase the spatial resolution by allowing the
separation and localization of these focal  sources.

\appendix

\subsection{Proof of Theorem~\ref{th:1D}}\label{pr:theorem1D}
\begin{proof}
~\\
    Let us first consider that the reconstructed signal is  $\mu^N=\alpha\delta_{x_K}$ with $a>0$.
    When $j=K$, $p_K=1$ in the optimality condition \eqref{eq:opt_l1} since $p_K\in\partial c_K$ and $a=c_K>0$.
    Thus, the optimality condition \eqref{eq:opt_l1}
    reduces to
    \begin{align}
        \lambda &=\gamma H(x_K-\xi) - aH(0) \label{eq:equality_for_c}
                \nonumber \\
         \Leftrightarrow \quad
         a &= \frac{\gamma H(x_K-\xi)-\lambda}{H(0)} \; .
    \end{align}
    Now let us consider the case where $j = K+1$, then \eqref{eq:opt_l1} becomes
    \begin{align} \label{eq:opt_taylor_case}
        \lambda p_{K+1} = \gamma H(x_{K+1}-\xi) - aH(x_{K+1}-x_K) \; .
    \end{align}
    Since, we assumed only one non-zero coefficient of $\mu^N$ at $x_K$,  we need to show that inequality condition $|p_{K+1}|<1$ holds.

    Inserting \eqref{eq:equality_for_c} into \eqref{eq:opt_taylor_case} yields
    \begin{align*}
        \lambda p_{K+1} = \gamma H(x_{K+1}-\xi) - \frac{\gamma H(x_K-\xi)-\lambda}{H(0)} H(h) \; ,
    \end{align*}
    where $h=|x_K-x_{K+1}|$.
    For this equation we consider the second order Taylor expansion of $H$ around
    zero.
    Note that $H'(0)=0$ holds, due to the maximum of $H$ at zero. Therefore, we
    obtain
    \begin{align*}
        \lambda p_{K+1} =\; &\gamma H(0) + \frac{\gamma }{2}H''(0)(x_{K+1}-\xi)^2 \\
        &- \left(\gamma + \frac{\gamma H''(0)}{2H(0)}(x_K-\xi)^2\right) \left(H(0) +
        \frac{1}{2}H''(0)h^2\right) \\
        &+ \lambda + \frac{\lambda}{2H(0)}H''(0)h^2 + \mathcal{O}(h^3) \; ,
    \end{align*}
    which reduces to
    \begin{align*}
        p_{K+1}& = 1 - \frac{\gamma}{2\lambda}H''(0)T+
        \mathcal{O}(h^3) \; ,
    \end{align*}
where $T=(x_K-\xi)^2 -
         (x_{K+1}-\xi)^2 +
h^2\!\left(1-\frac{\lambda}{\gamma H(0)}\right)$.
    Note also that $H''(0)<0$ as $H$ attains its maximum at zero and $\lambda<\gamma H(0)$.\\
    In order to obtain the inequality $|p_{K+1}|<1$, which would prove the
    assertion, $T$ has to be negative.
    This is true if and only if we have
    \begin{align*}
        x_K^2 - x_{K+1}^2 + 2\xi h < h^2\!\left(\frac{\lambda}{\gamma H(0)}-1\right) \; .
    \end{align*}
    This is equivalent to
    \begin{align*}
        2\xi h &< h^2\!\left(\frac{\lambda}{\gamma H(0)}-1\right) + (x_{K+1} + x_K)h \\
        \Leftrightarrow \hspace{1.35cm}
        \xi &< \frac{h}{2}\left(\frac{\lambda}{\gamma H(0)}-1\right) + \frac{1}{2}x_{K+1}
        + \frac{1}{2}x_K \\
        \Leftrightarrow \quad
        \xi - x_K &< \frac{h}{2}\left(\frac{\lambda}{\gamma H(0)}-1\right) + \frac{h}{2}
        \; .
    \end{align*}
    Thus, we obtain
    \begin{equation}
        \label{eq:optIneq_simple_1NonZeror}
        \xi - x_K < \frac{\lambda h}{2\gamma H(0)}  < \frac{1}{2}h \; ,
    \end{equation}
    which is true since
    we have $\lambda <\gamma  H(x_K-\xi)$
    and $H(x_K-\xi) < H(0)$.

Now assume that $h$ is sufficiently small in the latter case and
make the Ansatz $\mu^N = c_K \delta_{x_K\mu^N = } + c_{K+1}
\delta_{x_{K+1}}$. Without restriction of generality  we consider
$\gamma > 0$  hence we look for $c_K > 0$ and $c_{K+1} > 0$, the
other sign is analogous. We extend the vector $c$ by $c_j =0$ for
$j\notin \{K,K+1\}$ and verify that it is a minimizer of $J$ in
\eqref{eq:deconvolutionProblem} by constructing an appropriate
subgradient in the optimality condition  \eqref{eq:opt_l1_orig}.

In particular,  we
have
$$    p_j = \frac{1}\lambda [A^*(f-A c)]_j = p(x_j)$$
where
$$ p(x) =  \frac{\gamma}{\lambda} H(x-\xi) - \frac{c_K}{\lambda} H(x-x_K) - \frac{c_{K+1}}{\lambda} H(x-x_{K+1})\;. $$
The conditions $p_K = p_{K+1}  =1$ lead to the following $2\times 2$
system
\begin{align*}
\lambda &= \gamma H(x_K - \xi) - c_K H(0) - c_{K+1} H(h) \\
\lambda &= \gamma H(x_{K+1} - \xi) - c_K H(h) - c_{K+1} H(0)\;,
\end{align*}
for $c_K$ and $c_{K+1}$. If $h$ is sufficiently small, Taylor
expansion of $H$ around zero yields
\begin{align*}
&H(0) (c_{K} + c_{K+1}) +\frac{h^2}2 H''(0)  c_{K+1} = \\ &\qquad\lambda - \gamma H(0) - \gamma \frac{(x_K-\xi)^2}2 H''(0) +{\cal O}(h^3)  \\
&H(0) (c_{K} + c_{K+1}) +\frac{h^2}2 H''(0)  c_{K} = \\ &\qquad
\lambda - \gamma H(0) - \gamma \frac{(x_{K+1}-\xi)^2}2 H''(0) +{\cal
O}(h^3)\;.
 \end{align*}
From the leading  terms we obtain the solution
\begin{align*}
c_K &= \frac{1}2 \left(\gamma - \frac{\lambda}{H(0)} + \frac{x_K+x_{K+1} - 2\xi}h \right) + {\cal O}(h)\\
c_{K+1} &= \frac{1}2 \left(\gamma - \frac{\lambda}{H(0)} -
\frac{x_K+x_{K+1} - 2\xi}h \right) + {\cal O}(h)\;.
\end{align*}
Note also that $\gamma - c_K - c_{K+1} = \frac{\lambda}{H(0)} +
{\cal O}(h) $.

This implies for $x \notin [x_K,x_{K+1}]$
\begin{align*} p(x) &=  \frac{\gamma}{\lambda} H(x-\xi) - \frac{c_K}{\lambda} H(x-x_K) - \frac{c_{K+1}}{\lambda} H(x-x_{K+1}) \\
&= (\gamma - c_K - c_{K+1})   \frac{H(x-\xi) }{\lambda}  +
\frac{c_K}{\lambda} H'(x-\xi)(x_K - \xi)\\& +\frac{c_{K+1}}{\lambda}
H'(x-\xi)(x_{K+1} - \xi) + {\cal O}(h^2)\; .\end{align*}
For $x-\xi$
small we can again apply Taylor expansion around zero to show that $0
\leq p(x) \leq 1$.

For $x-\xi$ large we find
$$ p(x) = \frac{H(x-\xi) }{H(0)} + {\cal O}(h)\;, $$ and
since $0 \leq \frac{H(x-\xi) }{H(0)}  < 1$ we find $p(x) \in (-1,1)$
for grid size $h$ sufficiently small. Thus, the optimality condition is
satisfied on all grid points.
\end{proof}

\subsection{Adaptive superresolution approach:
implementation}\label{sec:Implementation}

In the following description, index $m=0,\ldots,$ denotes the
$m^\mathrm{th}$-iteration of the adaptive superresolution approach. In $m^\mathrm{th}$
iteration, the computational domain is denoted by
$\Omega^m\equiv(\mathcal{E}^m,\mathcal{N}^m)$ where
 $\mathcal{E}^m$ is the  set that includes all the elements  and
$\mathcal{N}^m=\{x_k^m\}_{1:N^m}$ is the set with the corresponding nodes that describe domain $\Omega^m$. 

  In $m^\mathrm{th}$ iteration
\begin{enumerate}
\item
         We solve the minimization problem
        \begin{equation}\label{eq:mRecursionMinimizationProblem}
              J_{\Omega^{m}}(c)= \frac{1}2\left(\sum_{j=1}^M
\sum_{k=1}^{N^m}c_k G(x_j-x_k^m)-w_j\right)^2 +\lambda\|c\|_1,
        \end{equation}
where the vector $c\in \mathbb{R}^{N^m}$ includes the nonzero
coefficients of the recovered signal $\mu^N$
\eqref{eq:numericalSolution} at nodes $x_k$ ,  $\|c\|_1
= \sum_{k=1}^{N^m} |c_k|$ and
$w\in\mathbb{R}^M$, i.e. $w =
          \{ f(z_j)+\varepsilon_j\}_{j=1,\ldots,M}$ is a sampled
        version of the observations $f$ on a set of measurement nodes
        $\{z_j\}_{j=1,\ldots,M}$ and $\varepsilon$ is the additive measurement noise.
\item Then, we update the computational domain.\\
   First, we remove the redundant elements (i.e. elements where
all their nodes correspond to zero entries in vector $c$).  We
define the new
domain $\Omega^m$ 
by estimating a set of disjoint clusters $\Omega_{\hat{l}}^m\equiv
(\mathcal{E}^m_{\hat{l}},\mathcal{N}^m_{\hat{l}})$ (groups of
adjacent elements) which comprises the remaining elements.
Hence, the update domain is
$\Omega^m=\bigcup_{\hat{l}=1}^{\hat{L}_m}\Omega_{\hat{L}}^m$ where
$\Omega^m_{\hat{l}}\bigcap_{\hat{l}=1}^{\hat{L}_m}=\emptyset$ and
$\hat{L}^m$ is the total number of formed clusters. 

  The mesh refinement is performed by including extra points/nodes in each
cluster $\Omega^m_{\hat{l}}$. 
The extra points/nodes  are at the centroids of the elements that comprise the clusters.
The centroids of very small elements are discarded.
The
choice of the centroids as extra nodes is based on the observation
that if $c_k\neq 0$ at node $x_k^m$, then the original peak should
be in the neighborhood of $x^m_k$ (stemming from the
analysis in section~\ref{sec:1Danalysis}). 
   For each cluster
$\Omega_{\hat{l}}^m$, a new set of  elements is estimated using the
updated set of nodes $\mathcal{N}^m_{\hat{l}}$ (old nodes and
centroids). Then, we repeat step 1, i.e. we solve
 problem (\ref{eq:mRecursionMinimizationProblem})
in the updated sets of nodes
$\mathcal{N}^{m}=\bigcup_{\hat{l}=1}^{\hat{L}^m}
\mathcal{N}^m_{\hat{l}}$.
\end{enumerate}

Steps 1-2 are repeated until  the computational support is not longer updated (the number of
nodes and elements stays fixed). This happens when the
distance between the nodes becomes small. As
a minimum distance (between two nodes) we can use a limit for peak separation
presented for some convolution kernels in
\cite{Poon2018,Poon2019}). Alternatively, prior information about the
expected size of the underlying peaks (e.g. in microscopy the sizes
of the molecules) can be considered.
Then, we recover as many peaks as the number of the disjoint
clusters
$\Omega_{\hat{l}}$, for $\hat{l}=1,\ldots,\hat{L}$
(where $\hat{L}$ is the total number of disjoint clusters estimated in the last iteration).
The amplitude denoted by $\hat{\gamma}_{\hat{l}}$ in cluster $\Omega_{\hat{l}}$  follows from equation~\eqref{eq:Peak_AmplitudeND} using the coefficients
entries of that cluster.  For the approximation of the
the peak location in cluster $\Omega_{\hat{l}}$, we first check the number of nonzero coefficients denoted by $N_{\hat{l}}$. If $N_{\hat{l}}>d$
(where $d$ is the dimensionality of the problem),
the peak location,
$\hat{\xi}_l$, is estimated by solving a linear system formed using the expression~\eqref{eq:PeakLoc_ND}.
Now, if $2\leq N_{\hat{l}}\leq d$ then the peak location can
be approximated with the help of linear basis functions,
$\phi_k$. If the approximated location is expressed as
$\hat{\xi}_{\hat{l}}=\sum_{k=1}^{N_{\hat{l}}}\phi_k x_k$
then by inserting the previous linear representation for
$\hat{\xi}_{\hat{l}}$ in expression~\eqref{eq:PeakLoc_ND}, we can obtain an
approximation for the peak location.
If $N_{\hat{l}}=1$, then $\hat{\xi}_{\hat{l}}$ equals to the value of the nonzero node (following from \eqref{eq:OptCond_TaylorExpansion}).

\section*{Acknowledgement}
AK was supported by the Academy of Finland Postdoctoral Researcher
program (No 316542).

\bibliographystyle{plain}
\bibliography{Run_eps_v2}

\end{document}